\documentclass[reqno,12pt]{amsart}
\RequirePackage[l2tabu, orthodox]{nag}
\usepackage{amsmath,amssymb,amsthm,enumerate,mathrsfs,colonequals}
\usepackage{fullpage,url,tikz,xspace,setspace}
\usepackage{microtype}
\usepackage{calc}
\usepackage{verbatim}
\usepackage[linesnumbered,ruled]{algorithm2e}
\usepackage[
]{hyperref}
\usepackage[alphabetic,lite,nobysame]{amsrefs} 

\def\be{\begin{equation}}
\def\ee{\end{equation}}
\def\ra{\rightarrow}
\def\llra{\longleftrightarrow}
\def\0{^{\phantom 0}}
\def\9{_{\phantom q}}
\def\m{{\phantom -}}

\DeclareMathOperator\Tr{Tr}
\DeclareMathOperator\val{val}

\DeclareMathOperator\FP{FP}
\DeclareMathOperator\Id{Id}
\def\ov{\overline}
\def\ul{\underline}

\newtheorem{theorem}{Theorem}[section]
\newtheorem{corollary}[theorem]{Corollary}
\newtheorem{lemma}[theorem]{Lemma}
\newtheorem{proposition}[theorem]{Proposition}
\newtheorem{ineq}[theorem]{INEQ}
\theoremstyle{definition}
\newtheorem{remark1}[theorem]{Remark}
\newtheorem{definition1}[theorem]{Definition}
\newtheorem{example}[theorem]{Example}
\newtheorem{problem}[theorem]{Problem}

\newcommand{\R}{\mathbf R}
\newcommand{\A}{\mathbb A}
\newcommand{\E}{\mathcal E}
\newcommand{\ox}{x}
\newcommand{\B}{\mathcal B}
\newcommand{\Q}{\mathbf Q}
\newcommand{\Z}{\mathbf Z}
\newcommand{\C}{\mathbf C}

\newcommand{\PP}{[b_1,\ldots , b_{N},\overline{a_{1},\ldots ,a_{k}}]}
\newcommand{\p}{(b_1,\ldots ,b_{N}, a_1, \ldots a_{k})}
\newcommand{\q}{\mathfrak q}
\newcommand{\OO}{\mathcal O}
\newcommand{\X}{y_1, \ldots, y_N, x_1,\ldots , x_k}

\newcommand{\BP}{\mathbb P} 

\DeclareMathOperator\Rea{Re}

\DeclareMathOperator\N{N}
\DeclareMathOperator\Gal{Gal}

\DeclareMathOperator\SL{SL}
\DeclareMathOperator\SLT{\SL_{2}}
\def\be{\begin{equation}}
\def\ee{\end{equation}}
\def\ra{\rightarrow}
\def\llra{\longleftrightarrow}
\def\0{^{\phantom 0}}
\def\9{_{\phantom q}}
\def\m{{\phantom -}}

\DeclareMathOperator\Quad{Quad}
\def\ov{\overline}
\def\ul{\underline}

\definecolorset{rgb}{}{}%
 {purpviolet,0.5,0,0.25}%

\newcommand{\ob}{\hat\beta}
\newcommand{\bs}{\beta^*}
\newcommand{\oA}{\overline{A}}

\DeclareMathOperator\GL{GL}

\begin{document}

\title{Periodic Continued Fractions over $\boldsymbol{S}$-Integers in Number Fields and Skolem's $\boldsymbol{p}$\,-adic Method }


\author{Bradley~W.~Brock}
\address{Center for Communications Research, 805 Bunn Drive, Princeton,
NJ 08540-1966, USA}
\email{brock@idaccr.org}

\author{Noam~D.~Elkies}
\address{Mathematics Department, Harvard University, 1 Oxford Street,
Cambridge, MA 02138-2901, USA}
\email{elkies@math.harvard.edu}

\author{Bruce~W.~Jordan}
\address{Department of Mathematics, Baruch College, The City University
of New York, One Bernard Baruch Way, New York, NY 10010-5526, USA}
\email{bruce.jordan@baruch.cuny.edu}

\begin{abstract}
We  generalize the classical theory of periodic continued fractions
(PCFs) over $\Z$ to rings 
$\OO$ of $S$-integers in a number field.
Let $\B=\{\beta, \bs\}$ be the multi-set of roots of a quadratic
polynomial in $\OO[x]$. We show that PCFs
$P=[b_1,\ldots,b_N,\overline{a_1\ldots ,a_k}]$ of type $(N,k)$
potentially converging  to a limit in $\B$ are given by
$\OO$-points on an affine variety $V:=V(\B)_{N,k}$ 
generically of dimension $N+k-2$.  We give the equations
of $V$  in terms of the
continuant polynomials of Wallis and Euler. 
The integral points $V(\OO)$
are related to writing matrices in $\SLT(\OO)$ as
products of elementary
matrices.  We 
give an algorithm to determine if a PCF
converges and, if so, to compute its limit.

Our standard example generalizes the PCF
$\sqrt{2}=[1,\overline{2}]$ to
the $\Z_2$-extension of $\Q$: $F_n=\Q(\alpha_n)$,
$\alpha_{n}:=2\cos(2\pi/2^{n+2})$, with integers $\OO_n=\Z[\alpha_n]$.
We want to find the PCFs of $\alpha_{n+1}$ over $\OO_{n}$ of type 
$(N,k)$ by
finding the $\OO_{n}$-points on $V(\B_{n+1})_{N,k}$
for $\B_{n+1}:=\{\alpha_{n+1}, -\alpha_{n+1}\}$.
There are three types $(N,k)=(0,3), (1,2), (2,1)$ such that the
associated PCF variety $V(\B)_{N,k}$ is a curve; we analyze these curves.
For generic $\B$,
Siegel's theorem implies that each of these three $V(\B)_{N,k}(\OO)$
is finite.
We find all the $\OO_n$-points on these PCF curves
$V(\B_{n+1})_{N,k}$ for $n=0,1$.
When $n=1$ we make extensive use of Skolem's $p$-adic method for $p=2$,
including its application to Ljunggren's equation $x^2 + 1 =2y^4$.

\end{abstract}

\maketitle


\section{Introduction}

A {\em regular} (or {\em simple}) continued fraction
$[c_1,c_2,c_3,\ldots]$ is one where $c_i\in \Z$ with $c_i>0$ for $i>0$.
It is a classical fact that all regular continued fractions converge.
Every irrational real number has a unique representation
as a regular continued fraction.
By a celebrated theorem of Lagrange \cite{lagrange}, a real number $\alpha$ is
a quadratic irrational if and only if
its regular continued fraction is eventually periodic, that is, of the form
$P=\PP$ .
The eventually periodic continued fraction $P$ -- henceforth simply
called peri\-odic -- is of {\em type} $(N,k)$ and of {\em period} $k$.
A famous example of this is the continued fraction
$\sqrt{2}=[1,\overline{2}]$ of type $(1,1)$.

In this paper we generalize the theory of periodic continued
fractions (PCFs) from $\Z$ to the $S$-integers $\OO$ of a number field $K$.
We establish the foundations of periodic $\OO$-continued fractions
in terms of products of $2\times 2$ matrices over~$\OO$.
A PCF $P$ over $\OO$ (or an $\OO$-PCF) formally satisfies a quadratic polynomial
$Q(x)\in\OO[x]$ with multi-set of roots
$\B=\{\beta, \bs\}\subseteq \BP^1(\overline{\Q})$.
We say that $P$ has roots $\B$.
If $P$ actually converges to a limit $\hat{\beta}$, then $\hat{\beta}\in\B$.
Requiring that $P=[y_1,\ldots, y_N,\overline{x_1,\ldots , x_k}]$
formally satisfy 
$Q(x)=Ax^2+Bx+C\in\OO[x]$
with roots $\B$
defines a {\em PCF variety} 
\mbox{$V:=V(\B)_{N,k}\subseteq \A^{N,k}:=\A^N\times
\A^k$}  generically of dimension \mbox{$N+k-2$}.  
If $\alpha\in\OO$ but $\beta=\sqrt{\alpha}\notin \OO$, we write $V(\alpha)_{N,k}$ for $V(\{\beta,-\beta\})_{N,k}$.
We give the equations of these PCF varieties using the
continuant polynomials of Wallis \cite{wallis}*{p.~191} and Euler \cite{euler}*{pp.~103,106}.
An $\OO$-PCF $\PP$ with roots $\B$  corresponds to a point \mbox{$\p\in V(\OO)$}.
Unlike the classical theory of regular $\Z$-PCFs,
here $\beta$ can have infinitely many $\OO$-PCFs of type
$(N,k)$. This
leads to questions about the integral points on the varieties $V(\B)_{N,k}$ --
for example, are the $\OO$-points degenerate,
i.e., is their Zariski closure a proper subvariety?
As a first result on the geometry of $V(\B)_{N,k}$, we show that
it fibers over Fermat-Pell conic curves in Theorem \ref{FPfibration}.

But there is more than the algebraic theory: one also has to worry about
convergence, which is delicate.
For example, $[1,\overline{-1,2}]$ does not
converge, but $[1,\overline{-2,2}]$ does.
We formulate classical convergence conditions
as Algorithm \ref{con} which is easy to apply, deferring the proofs to an appendix.
The criteria for convergence involve both algebraic conditions and inequalities.
For {\em any \/}$\B$ we define an affine variety $V_{N,k}$ in Definition
\ref{pint} which is a divisor on $V(\B)_{N,k}$ having the property that all
points on $V_{N,k}$ correspond to {\em divergent} continued fractions,
cf. Corollary \ref{uber}.
Finding $\OO$-PCFs for $\beta$ then
entails a two-step process: 1) Find the points $(b_1,\ldots , b_N,
a_1,\ldots , a_k)$ on $V(\B)_{N,k}(\OO)$.
2) Determine whether the $\OO$-PCF $[b_1,\ldots
,b_N,\overline{a_1,\ldots ,a_k}]$ converges to $\beta$ 
using Algorithm \ref{con}.

We generalize the prototypical continued fraction
$\sqrt{2}=[1,\overline{2}]$, which gives a $\Z$-point on
$V(2)_{1,1}$, to the $\Z_2$-extension $\bigcup_{n\geq 0} F_n$ of $\Q$
with $F_n=\Q(\alpha_n)$ and $\alpha_{n}=2\cos(2\pi/2^{n+2})$.
Note that $\alpha_{n+1}^2=2+\alpha_{n}$ and that
the integers of $F_{n}$ are $\OO_{n}=\Z[\alpha_{n}]$ for $n\geq 0$.
Our problem is to find the $\Z[\alpha_n]$-PCFs of type $(N,k)$
for  $\alpha_{n+1}$, with $\sqrt{2}=[1,\overline{2}]$ corresponding to
$n=0$ and $(N,k)=(1,1)$.  Hence we have:
\begin{problem}
\label{problem}
Find $V(2+\alpha_n)_{N,k}(\OO_n)$, $n\geq 0$.
\end{problem}
\noindent In Section \ref{drive} we analyze the easy case of the three 
PCF varieties of dimension less than $1$, solving Problem \ref{problem}
for
$(N,k)=(0.1),\,(0,2),\, (1,1)$.

The second half of the paper -- Sections \ref{toady1}, \ref{ghost},
and \ref{sec:integral} -- studies the diophantine geometry of the three PCF varieties
which are curves, namely $V(\B)_{N,k}$ with $(N,k)=(0,3)$, $(2,1)$, and $(1,2)$.
We prove in Theorems~\ref{great}, \ref{rentdue}, and~\ref{tadpole}
that for generic $\B$ the $\OO$-points
on these PCF curves are finite in number by applying Siegel's theorem.
We solve Problem \ref{problem} for
$(N,k)= (2,1)$, $(1,2)$, and $(0,3)$ (the cases of curves) and $n=0,1$.
These curve examples amply illustrate the arithmetic richness of these varieties,
with the case $n=1$ not surprisingly giving the greatest difficulty.
The one break we get is that $V(2+\sqrt{2})_{2,1}(\Z[\sqrt{2}])=\emptyset$ by a
simple congruence argument -- so there are no periodic $\Z[\sqrt{2}]$-continued
fractions
of $\sqrt{2+\sqrt{2}}$ of type $(2,1)$, cf. Proposition \ref{grunt}.
We use Skolem's $p$-adic method \cites{skolem2,skolem} for $p=2$ to find
$V(2+\sqrt{2})_{0,3}(\Z[\sqrt{2}])$ and $V(2+\sqrt{2})_{1,2}(\Z[\sqrt{2}])$.
Skolem's method does not always apply in diophantine problems, and it does not
always work even when it applies, but on these PCF curves over the
$\Z_2$-extension of $\Q$ it is effective.
The argument for $V(2+\sqrt{2})_{1,2}$ in Section \ref{sec:integral} is
particularly involved since Skolem does not apply directly
but only after passing to a cover.

An early (1942) application of the $p$-adic method was 
Ljunggren's famous result \cite{ljunggren} solving
$x^2+1 =2y^4$ over $\Z$ -- the only solutions are 
$(x,y)= (\pm 1,\pm 1)$ and $(x,y)=(\pm 239, \pm 13)$.
We use Ljunggren two separate times--once for
$V(2+\sqrt{2})_{0,3}$ and once for $V(2+\sqrt{2})_{1,2}$.
Indeed, the appearance of $239+169\sqrt{2}$ in the periodic continued fraction 
\[
\sqrt{2+\sqrt{2}}=[\overline{57-39\sqrt{2}, 239+169\sqrt{2}, -733+52\sqrt{2}}]
\]
of type $(0,3)$ in Corollary \ref{rinds} is certainly suggestive of Ljunggren.
Theorem \ref{trails} gives the  sixteen  $\Z[\sqrt{2}]$-points on $V(2+\sqrt{2})_{0,3}$ 
with the eight $\Z[\sqrt{2}]$-PCFs of $\alpha_2=\sqrt{2+\sqrt{2}}$
of type $(0,3)$ given in Corollary \ref{rinds}.
Likewise Theorem \ref{tips} finds the twenty $\Z[\sqrt{2}]$-points on (a component of)
$V(2+\sqrt{2})_{1,2}$ with the ten $\Z[\sqrt{2}]$-PCFs
for $\alpha_{2}$ of type $(1,2)$ given in Corollary \ref{pot}.
In each case we prove there are no others.
The integral points we find are interesting: one
could hardly guess the very slowly converging periodic continued fraction
\begin{equation*}
    \sqrt{2+\sqrt{2}}=\left[442+312\sqrt{2},\,\overline{-298532+211094\sqrt{2},\,884+624\sqrt{2}}\right]
\end{equation*}
of type $(1,2)$ appearing in Corollary \ref{pot}.

\section{Periodic Continued Fractions}

Suppose $c_i\in\C$, $i\ge 1$.
A finite continued fraction
$F:=[c_1,c_2, c_3,\ldots, c_n]$ with {\em partial quotients} $c_i$
is defined inductively as follows.
\begin{align}
  [c_1]&=c_1,\nonumber\\
  [c_1,c_2]&=c_1 + 1/c_2,\text{ and}\nonumber\\
  [c_1,c_2,c_3,\ldots, c_n]&=[c_1,[c_2, c_3, \ldots , c_n]]\nonumber\\
&=c_1+\cfrac{1}{c_2+\cfrac{1}{c_3+\cfrac{1}{c_4
+\cfrac{1}{\raisebox{-1.2em}{\ensuremath{\ddots\quad}}\raisebox{-2.4em}{\ensuremath{c_{n-1}+\cfrac{1}{c_n}}}}}}}
\in\BP^1(\C)\label{peanut}.
\end{align}
In other words, if we define an automorphism of $\BP^1$ by
$\phi_c(z)=c+1/z$ and put
\begin{equation*}
\phi_F:=\phi_{c_1}\circ\phi_{c_2}\circ\cdots\circ\phi_{c_n},
\end{equation*}
then $F=\phi_F(\infty)$.
For $\alpha\in\C$ define the matrix
\begin{equation}
\label{goon}
D(\alpha):=\left[\begin{matrix}\alpha &1\\1&0\end{matrix}
\right],\quad\mbox{and set }t:=D(0);\quad\mbox{note that }
D(\alpha)^{-1}=tD(-\alpha)t=\left[\begin{matrix}0&1\\1&-\alpha \end{matrix}\right].
\end{equation}

For a matrix $A=\left[\begin{smallmatrix}a&b\\c&d\end{smallmatrix}\right]\in\GL(2,\C)$ define
the automorphism $\oA$ of $\BP^1$ by
\[
\oA:z\mapsto \frac{az+b}{cz+d}
\]
for $z\in\BP^1(\C)$.
With this notation
\[
\phi_c=\overline{D(c)} .
\]
For $\beta\in\C$, set
\begin{equation*}
\vec{v}_{\beta}:=\left[\begin{matrix}\beta\\1\end{matrix}\right]
\quad\mbox{and}\quad
\vec{v}_{\infty}:=\left[\begin{matrix}1\\0\end{matrix}\right].
\end{equation*}
As usual we identify $z\in\BP^{1}(\C)$ with $\vec{v}_{z}$.

\begin{definition1}\label{hartebeest}

Let $F=[c_1,c_2, c_3,\ldots, c_n]$ be a finite continued fraction.
Define
\begin{equation*}
M(F)=\left[\begin{matrix}p_n&p_{n-1}\\q_n&q_{n-1}\end{matrix}\right]
=\prod_{i=1}^nD(c_i)=D(c_1)D(c_2)\cdots D(c_n).
\end{equation*}
\end{definition1}
\noindent Explicit formulas for $p_n$ and $q_n$ are known classically;
we give them later in \eqref{jello}.

 So $\phi_F=\overline{M(F)}$,
and $F=p_n/q_n$ in terms of the $c_i$.
Furthermore, using \eqref{goon}
\begin{align*}
M(F)^{-1}&=\prod_{i=0}^{n-1}[tD(-c_{n-i})t]
=t\left[\prod_{i=0}^{n-1}D(-c_{n-i})\right]t\\ 
&=M([0,-c_n,\ldots,-c_1,0])\nonumber
\end{align*}
gives the inverse.
Note that
\begin{equation*}
\det M(F)=\det M(F)^{-1}=(-1)^{n}.
\end{equation*}

We define the infinite continued fraction $C:=[c_1,c_2, c_3,c_4,\ldots]$
to be a formal expression as in \eqref{peanut}, but which does not terminate.
We define $M_n(C)=M([c_1,c_2, c_3,\ldots,c_n])$ for $n\ge1$, $M_0(C)=I$,
and its {\em convergents} to be
$[c_1,c_2, c_3,\ldots,c_n]=M_n(C)_{11}/M_n(C)_{21}\in\BP^1(\C).$
The {\em value} of $C$ is
\begin{equation*}
 \ob([c_1,c_2,c_3,c_4,\ldots])=\lim_{n\to\infty}[c_1,c_2,c_3,\ldots, c_n]
\end{equation*}
if the limit exists.
In the theory of {\em regular} (or {\em simple}) continued fractions for
approximating real numbers, the $c_i\in\Z$ are positive for $i>1$,
and the limit always exists.
We say that $C$ is over a ring $\OO\subseteq \C$
(or that it is an $\OO$-continued fraction)
if $c_i\in\OO$ for all $i$.

The symbol $[c_1,c_2, c_3, \ldots]$ is often used
to denote both the continued fraction and its value if it exists, but
we have chosen to restrict this abuse of notation to the finite case
or when all the $c_i$'s are explicit real numbers.
In addition we write equality between two continued
fractions if they are connected by formal manipulations that leave the
value unchanged if it exists.
For example, the following lemma says that we can write
\begin{equation*}
  [\ldots,a,0,b,\ldots]=[\ldots,a+b,\ldots].
\end{equation*}
\begin{lemma}
\label{caribou}
Let $C=[\ldots,a,0,b,\ldots]$ be a finite or infinite continued fraction,
one of whose partial quotients is $0$.
The value $\ob(C)$ exists if and only if $\ob([\ldots,a+b,\ldots])$ exists,
in which case the two values are equal.
\end{lemma}
\begin{proof}
It suffices to observe that $\phi_a(\phi_0(\phi_b(z)))=\phi_{a+b}(z)$.
\end{proof}
\noindent Another way to prove this is to note that the convergents of
$[\ldots,a+b,\ldots]$ are the same as those of $[\ldots,a,0,b,\ldots]$
with two omissions.

\begin{definition1}

The infinite continued fraction $P=[c_1,c_2,c_3, \ldots]$ is {\em periodic}
if there are integers $N\ge0$ and $k\ge1$ so that $c_n=c_{n+k}$ for all
integers $n> N$.
We write
\begin{equation}
\label{gorilla}
  P:=[b_1,b_2,\ldots, b_{N}, \overline{a_{1},a_{2},\ldots, a_{k}}]
\end{equation}
for the infinite periodic continued fraction 
\[
[b_1,\ldots, b_N, a_1, \dots, a_k,a_1, \ldots, a_k, a_1,\ldots, a_k, \ldots ]=[c_1,c_2, c_3, \ldots ]
\]
if $c_n$ depends only on $n\bmod k$ for $n>N$,
and say $P$ has {\em type} $(N,k)$ and {\em period} $k$.
If $N$ and $k$ are both minimal, we say that $(N,k)$ is the
{\em minimal type} and $k$ is the {\em minimal period}.
We set $\A^{N,k}=\A^N\!\times\! \A^k\cong \A^{N+k}$.
The PCF $P$ in \eqref{gorilla}
determines a point $p=p(P)=\p\in\A^{N,k}$ and likewise
the point $p=\p\in \A^{N,k}$ determines the PCF
$P=P_{N,k}(p)=\PP$.
\end{definition1}

\begin{definition1}
\label{muskdeer}

  We set the following notation for
\[
P=\PP\quad\mbox{or}\quad p=\p\in\A^{N,k}:
\]
\begin{align}
  E_{N,k}(p)=E(P)&=\left[\begin{matrix} E_{11}(P)& E_{12}(P)\\ E_{21}(P) & E_{22}(P)\end{matrix}\right]
 = M([b_1,\ldots,b_N, a_{1}, \ldots, a_{k},0,-b_{N},\ldots,-b_1,0])\label{fed1}\\
 &= D(b_1)\cdots D(b_N)D(a_1)\cdots  D(a_{k})tD(-b_N)\cdots D(-b_1)t\nonumber\\
&=M([b_1,\ldots,b_{N}])M([a_{1},\ldots,a_{k}])M([b_1,\ldots,b_{N}])^{-1},
\nonumber\\
\Quad(P)(x)&:=\Quad_{N,k}(p)(x) = E_{21}(P)x^2+[E_{22}-E_{11}](P)x-E_{12}(P)
\mbox{ having }\nonumber\\ 
\B(P)&:=\B_{N,k}(p) =\{\beta,\bs\}\mbox{ as multi-set of roots so that }\nonumber\\
\Quad(P)(x)&=E_{21}(P)(x-\beta)(x-\bs)\mbox{ if }\infty\not\in\B(P).\nonumber\\
 P^*&=[b_1,b_2,\ldots, b_{N},0, \overline{-a_{k},-a_{k-1},\ldots,-a_{1}}]\nonumber\\
  &\hskip-4.5em =\begin{cases}[0, \overline{-a_{k},-a_{k-1},\ldots,-a_1}]&\text{if }N=0\nonumber\\
  [b_1,b_2,\ldots,b_{N-1}, b_{N}-a_{k},\overline{-a_{k-1},\ldots,-a_{1},-a_{k}}]&\text{if }N>0
\end{cases}\nonumber
\end{align}
using Lemma \ref{caribou}. We have $E_{11}(P), E_{12}(P), E_{21}(P), E_{22}(P)\in
\Z[b_1,\ldots ,b_N, a_1, \ldots a_k]$.

In terms of variables $(y_1, \ldots, y_N, x_1, \ldots, x_k)$ put
\begin{equation}
\label{foggy}
E_{N,k}=\left[\begin{matrix} E_{11}& E_{12}\\ E_{21} & E_{22}\end{matrix}\right]=
\begin{cases} M([y_1, \ldots, y_N])M([x_1,\ldots, x_k])M([y_1, \ldots, y_N])^{-1}\\
D(y_1)\cdots D(y_N)D(x_1)\cdots D(x_k)tD(-y_N)\cdots D(-y_1)t;
\end{cases}
\end{equation}
we have $E_{11}, E_{12}, E_{21}, E_{22}\in\Z[y_1, \ldots, y_N,x_1,\ldots, x_k]$.
\end{definition1}

\begin{remark1}
\label{rindss}
Note that $\det E_{N,k}(p)=(-1)^k$.
\end{remark1}

\noindent We call $\B=\B_{N,k}(P)$ the {\em roots} of $P$ or $p$.
Here we have the usual convention that $\infty$ is a root of $0x^2+Bx+C$
and a double root of $0x^2+0x+C$.
We leave $\B_{N,k}(P)$ undefined if $E_{N,k}(P)$
is a multiple of the
identity because this corresponds to $\Quad_{N,k}(P)=0$.
In Proposition \ref{ratify} we shall see that $E(P)$ does not depend on
the choice of $N$ but does depend on $k$.
In particular it can happen that $E(P)$ is defined for $k$ but not defined
for some multiple of $k$, corresponding to the case $G=0$
in Proposition \ref{ratify}.
This can only happen if $\ob(P)$ does not exist.
The PCF
$P^*$ is called the {\em dual} of $P$ (Galois \cite{galois}).
So if $N>0$, the dual of a PCF of type $(N,k)$
can be made to be of type $(N,k)$.
However, the dual of type $(0,k)$ is of type $(1,k)$.
The dual of the dual gives the original continued fraction because
$[\ldots,a,0,0,b,\ldots]=[\ldots,a,b,\ldots]$ by Lemma \ref{caribou}.
We shall see that the value $\hat{\beta}(P)$ of a 
PCF if it exists is a
root of the quadratic polynomial $\Quad  (P)$, and the value of the dual if it exists is the
other root.
Note that
\begin{align}
E(P^*)&=M([b_1,\ldots,b_{N}])M([a_{1},\ldots,a_{k}])^{-1}
  M([b_1,\ldots,b_{N}])^{-1}\nonumber \\
&=E(P)^{-1}=
(-1)^k\begin{bmatrix}E_{22}(P)& -E_{12}(P)\\ -E_{21}(P) & E_{11}(P)
\end{bmatrix}.\label{sitatunga}
  \end{align}
Consequently, $\Quad_{N,k}(P^*)=-(-1)^k\Quad_{N,k}(P)$ and $\B(P^*)=\B(P)$.

The matrix $E(P)$ plays a key role; it is almost true that the conjugacy
class of $E(P)$ determines the convergence behavior of $P$, the exception being
Theorem \ref{bactrian}\ref{bactrian2}.
The following straightforward linear algebra proposition applies to a
general $2\times2$ matrix $E$, which we shall apply to our matrix $E(P)$.

\begin{proposition}
\label{returned}
Let
$E=\left[\begin{smallmatrix}a&b\\c&d\end{smallmatrix}\right]$
with at least one of $a-d$, $b$, $c$ nonzero.
If $\beta$ is a root of 
\[
cx^2+(d-a)x-b=0 ,
\]
possibly infinite, then
the vector $\vec{v}_{\beta}$ is the unique eigenvector of $E$ up to scalars
with eigenvalue $a$ if $\beta=\infty$ and with eigenvalue $c\beta+d$ otherwise.
\end{proposition}

\begin{example}
Suppose $P=[2,\overline{-2,4}]=\sqrt{2}$ with $(N,k)=(1,2)$.
Then
\begin{align*}
& E:=E(P)=\left[\begin{matrix} -3& -4\\ -2 & -3\end{matrix}\right]
=\begin{bmatrix} 2& 1\\ 1 & 0\end{bmatrix}
\begin{bmatrix} -2& 1\\ 1 & 0\end{bmatrix}
\begin{bmatrix} 4& 1\\ 1 & 0\end{bmatrix}
\begin{bmatrix} 2& 1\\ 1 & 0\end{bmatrix}^{-1},\\
\Quad_{1,2}(P)&=-2x^2+4,\ \B_{1,2}(P)=\{\sqrt{2},-\sqrt{2}\},\  P^*=[-2,\overline{2,-4}], \
\det E =1=(-1)^k,\\
&\hspace*{.38in}  E\vec{v}_{\!\mathsmaller{\sqrt{2}}}=(-2\sqrt{2}-3)
\vec{v}_{\!\mathsmaller{\sqrt{2}}},\ \mbox{and }  
E\vec{v}_{\!\mathsmaller{-\sqrt{2}}}=(2\sqrt{2}-3)
\vec{v}_{\!\mathsmaller{-\sqrt{2}}}.
\end{align*}
\end{example}

For any continued fraction $C=[c_1,c_2, c_3, \ldots]$ formal manipulation
shows that
\[[0,-c_{N},-c_{N-1},\ldots,-c_2,-c_1,0,C]=[c_{N+1},c_{N+2},c_{N+3},\ldots].\]
Hence, since our continued fraction $P$ is periodic of type $(N,k)$ we have
\begin{equation*}
P  = \left[b_1,\ldots, b_k,a_1, \ldots a_k,0,-b_{N},-b_{N-1},
\ldots,-b_2,-b_1,0,P\right].
\end{equation*}
Thus $P=\overline{E(P)}(P)$ by Equation \eqref{fed1}.  Since
\[
E(P)=E_{N,k}(P)=\begin{bmatrix}E_{11}(P)&E_{12}(P)\\E_{21}(P) & E_{22}(P)\end{bmatrix},
\]
we get
\[
\frac{E_{11}(P)P+E_{12}(P)}{E_{21}(P)P+E_{22}(P)}=P\quad\mbox{and hence}\quad
E_{21}(P)P^2+[E_{22}-E_{11}](P)P-E_{12}(P)=0.
\]
So $\Quad_{N,k}(P)(P)=0$
provided $\overline{E(P)}$ is not the identity.

We can write the entries of $E(P)$
explicitly using Euler's continuant polynomials  $K_n$ \cite{euler}.
Define $K_n$ recursively by
\begin{align}
K_{-2}=1,\quad K_{-1} &= 0, \quad K_0 =1,\nonumber\\
K_1(c_1)&=c_1,\nonumber\\
K_n(c_1,\ldots,c_n)&=K_{n-1}(c_1,\ldots,c_{n-1})c_n+K_{n-2}(c_1,\ldots,c_{n-2})\label{dog}\\
\text{or, equivalently, }K_n(c_1,\ldots,c_n)&=c_1K_{n-1}(c_2,\ldots,c_n)\nonumber
 +K_{n-2}(c_3,\ldots,c_n).
\end{align}
For example,
\begin{align*}
K_2(c_1,c_2)&=c_1c_2+1,\\
K_3(c_1,c_2,c_3)&=c_1c_2c_3+c_1+c_3,\\
K_4(c_1,c_2,c_3,c_4)&=c_1c_2c_3c_4+c_1c_2+c_1c_4+c_3c_4+1,\\
K_5(c_1,c_2,c_3,c_4,c_5)&=c_1c_2c_3c_4c_5+c_1c_2c_3+c_1c_2c_5+c_1c_4c_5+c_3c_4c_5\\
&\quad+c_1+c_3+c_5,\text{ etc.}
\end{align*}
The recursion relation \eqref{dog} is exactly what is needed to show
the identity
\begin{equation}
\label{lastly}
D(c_1)D(c_2)\cdots D(c_n)=\begin{bmatrix} K_n(c_1,\ldots c_n)& K_{n-1}(c_1,\ldots, c_{n-1})\\
K_{n-1}(c_2,\ldots, c_n)&K_{m-2}(c_2,\ldots,c_{n-1})\end{bmatrix}
\end{equation}
for $n\geq 0$ by induction.
This in turn gives formulas for $p_n$ and $q_n$ in Definition \ref{hartebeest}:
\begin{equation}
\label{jello}
p_n=K_{n}(c_1,\ldots,c_n)\quad\mbox{and}\quad q_n=K_{n-1}(c_2,\ldots,c_n).
\end{equation}
Continuant polynomials have the following properties:
\begin{align}
[c_1,c_2,\ldots, c_n]&=\frac{K_{n}(c_1,c_2,\ldots,c_n)}{K_{n-1}(c_2,c_3,\ldots, c_n)},\label{llama}\\
K_{n+1}(c_1,\ldots,c_n,0)&=K_{n-1}(c_1,\ldots,c_{n-1}),\nonumber \\ 
K_{m+n+1}(a_1,\ldots,a_m,0,b_1,\ldots,b_n)&=K_{m+n-1}(a_1,\ldots,a_m+b_1,\ldots,b_n)\label{guanaco}\\
=K_{m-1}(a_1,\ldots,a_{m-1})K_n(b_1,\ldots,b_n)&+
K_m(a_1,\ldots,a_m)K_{n-1}(b_2,\ldots,b_n).
\nonumber
\end{align}
Equations \eqref{llama} and \eqref{guanaco} give another proof of
Lemma \ref{caribou}.

We can use \eqref{lastly} to give the $2\times2$ matrix $E_{N,k}(P)$ in \eqref{fed1}:
\begin{proposition}
\label{rounded}
We have $E_{N,k}(P)
=\begin{bmatrix}E_{11}(P)&E_{12}(P)\\E_{21}(P) & E_{22}(P)\end{bmatrix}$  with
\begin{align*}
E_{11}(P)&=K_{2N+k+2}(b_1,\ldots, b_N, a_1, \ldots , a_{k},0,-b_{N},\ldots,-b_1,0)\nonumber\\
&=\begin{cases}
K_{k}(a_1,\ldots,a_{k})&\text{if }N=0,1,\\
K_{2N+k-2}(b_1,\ldots, b_N, a_1,\ldots a_{k}-b_{N},-b_{N-1}, \ldots,-b_2)&\text{if }N\ge2;
\end{cases}\\ 
E_{21}(P)&=K_{2N+k+1}(b_2,\ldots, b_N, a_1, \ldots, a_{k},0,-b_{N},\ldots,-b_1,0)\nonumber\\
&=\begin{cases}
K_{k-1}(a_2,\ldots,a_{k})&\text{if }N=0,1,\\
K_{2N+k-3}(b_2,\ldots , b_N, a_1, \ldots ,a_{k}-b_{N}, -b_{N-1}, \ldots,-b_2)
&\text{if }N\ge 2;
\end{cases}\\ 
E_{12}(P)&=K_{2N+k+1}(b_1,\ldots, b_N, a_1, \ldots, a_{k},0,-b_{N},\ldots,-b_1)
\nonumber\\
&=\begin{cases}
K_{k-1}(a_1,\ldots,a_{k-1})&\text{if }N=0,\\
K_{k+1}(b_1,a_1, a_2, \ldots,a_{k}-b_1)&\text{if }N=1,\\
K_{2N+k-1}(b_1,\ldots, b_N, a_1, \ldots,   a_{k}-b_{N},-b_{N-1}, \ldots,-b_1)
&\text{if }N\ge2;
\end{cases}\\ 
E_{22}(P)&=K_{2N+k}(b_2,\ldots , b_N, a_1,\ldots , a_{k},0,-b_{N},\ldots,-b_1)\nonumber\\
&=\begin{cases}
K_{k-2}(a_2,\ldots,a_{k-1})&\text{if }N=0,\\
K_{k}(a_1,\ldots, a_{k-1}, a_{k}-b_1)&\text{if }N=1,\\
K_{2N+k-2}(b_2,\ldots,b_N, a_1, \ldots, a_{k}-b_{N},-b_{N-1}, \ldots,-b_1)&
\text{if }N\ge2.
\end{cases}
\end{align*}
\end{proposition}
Once we establish that $P$ does not converge if $E$ is a multiple of the
identity in Theorem \ref{bactrian},
the following proposition will follow from the above discussion.
\begin{proposition}
If $P=\PP$
converges, then its value $\ob(P)\in\B_{N,k}(P)$ for any type
$(N,k)$ of $P$.
If $P$ is over $\OO$,
then $\ob(P)$ satisfies the quadratic polynomial $\Quad (P)\in\OO[x]$.
\end{proposition}

The next proposition gives a property of the entries of $E$
which will have a geometric consequence in Section \ref{varieties}.
To help with the notation make the $a_i$ periodic by writing
\begin{equation}
\label{periodic}
a_i=a_{1+((i-1)\bmod k)} \mbox{ for } i>k.
\end{equation}

\begin{proposition}
\label{ratify}
Fix integers $N,\ell\ge 0$ and $k,m\ge 1$.
Define
\begin{align}
  d_i&=\begin{cases}
 b_i&\text{for }1\le i\le N,\\
 a_{i-N}&\text{for }N+1\le i\le N+\ell,
\end{cases}\label{onto}\\
c_i&=a_{i+\ell}\text{ for }1\le i\le mk,\nonumber
\end{align}
and define $G=G_{k,m}(a_{1},\ldots, a_{k})\in\Z[a_{1},\ldots,a_{k}]$ by
\begin{align*}
G=\sum_{j=0}^{\lfloor(m-1)/2\rfloor}&(-1)^{(k+1)j}\binom{m-1-j}{j}\times\\
&(K_k(a_{1},\ldots,a_{k})+K_{k-2}(a_{2},\ldots,a_{k-1}))^{m-1-2j}.
\end{align*}
Then\begin{align*}
(E_{N+\ell, mk})_{21}(d_1,\ldots,d_{N+\ell},c_1,\ldots,c_{mk})&=
G(E_{N,k})_{21}(b_1,\ldots,b_N,a_1,\ldots,a_{k}),\\
[(E_{N+\ell, mk})_{22}-(E_{N+\ell,
mk})_{11}](d_1,\ldots,d_{N+\ell},& c_1,\ldots,c_{mk})
\\
= G[(E_{N,k})_{22}&-(E_{N,k})_{11}](b_1,\ldots,b_N,a_1,\ldots,a_{k}),\\
(E_{N+\ell, mk})_{12}(d_1,\ldots,d_{N+\ell},c_1,\ldots,c_{mk})&=
G(E_{N,k})_{12}(b_1,\ldots,b_N,a_1,\ldots,a_{k}).
\end{align*}
\end{proposition}
\begin{proof}

First we prove the proposition for $m=1$.
Since $a_{k+i}=a_{i}$ for $1\le i\le\ell$, we find
\begin{align*}
E_{N+\ell,k}&=M([b_1,\ldots,b_N,a_1,\ldots,a_{k+\ell},0,-a_\ell,\ldots,-a_1,-b_N,\ldots,-b_1,0])\\
          &=M([b_1,\ldots,b_N,a_1,\ldots,a_{k},0,-b_N,\ldots,-b_1,0])=E_{N,k}
\end{align*}
by applying the identity \eqref{guanaco} $\ell$ times, thus proving $G=1$
in this case.

So without loss of generality we can now assume $\ell=0$.
Let $E$ be as in Equation \eqref{fed1}
so that $E_{N,mk}=(E_{N,k})^m$, and
\[
E_{11}+E_{22}=\Tr(E)=K_k(a_{1},\ldots,a_{k})+K_{k-2}(a_{2},\ldots,a_{k-1})
=r_1+r_2,
\]
where $r_i$ are the eigenvalues of $E$,
since trace is invariant under conjugation.
By the Cayley--Hamilton theorem and induction,
$E^m=t_{m-1}E-\det(E)t_{m-2}I$ where
\[t_{m-1}=\begin{cases}
\frac{r_1^m-r_2^m}{r_1-r_2}&\text{if }
r_1\ne r_2,\\
mr_1^{m-1}&\text{if }r_1=r_2.
\end{cases},\text{ so }
\frac{(E_{N,mk})_{21}}{(E_{N,k})_{21}}=\frac{(E_{N,mk})_{22}-(E_{N,mk})_{11}}{(E_{N,k})_{22}-(E_{N,k})_{11}}=\frac{(E_{N,mk})_{12}}{(E_{N,k})_{12}}=G
\]
with $G:=t_{m-1}$.
From this, together with $\Tr(E)=r_1+r_2$
and $\det(E)=r_1r_2=(-1)^k$, we can derive the
polynomial expression for $G$
to be $U_{m-1}(\Tr(E)/2)$ if $k$ is even
and $U_{m-1}(i\Tr(E)/2)/i^{m-1}$ if $k$ is odd,
where $U_{m-1}$ is the Chebyshev polynomial of the second kind.
\end{proof}

\section{PCF Varieties}
\label{varieties}

Let $\overline{\Q}\hookrightarrow\C$ be an algebraic closure of $\Q$, and let
$\OO\subseteq \overline{\Q}$ be the $S$-integers in a number
field $K$.
Suppose $\B$ is the multi-set $\{\beta, \, \bs\}$ of roots in $\BP^1(\overline{\Q})$
of the quadratic polynomial $Ax^2+Bx+C\in\OO[x]$ 
with $(0,0,0)\ne(A,B,C)\in\BP^2(\OO)$ and the usual convention that one
(resp., both) of $\beta,\bs$ is $\infty$ if $A=0$ (resp., $A=B=0$).
Of course $Ax^2+Bx+C\in\OO[x]=A(x-\beta)(x-\bs)$ if $A\neq 0$.
\begin{definition1}
\label{pint}
Let $E_{N,k}=\begin{bmatrix} E_{11} & E_{12}\\
E_{21} & E_{22}\end{bmatrix}$.  
Define the variety
\begin{align}
V_{N,k} : E_{21}(\X )=&
[E_{22}-E_{11}](\X )\label{toady}\\
=& E_{12}(\X ) =0.\nonumber
\end{align}
The {\em PCF variety}  $V(\B)_{N,k}$ is the affine variety
over $\OO$ defined by the three equations
\begin{align}
A[E_{22}-E_{11}](\X ) & = BE_{21}(\X ),\label{foggy1}\\
- AE_{12}(\X ) & = CE_{21}(\X ),\nonumber\\
- BE_{12}(\X ) & = C[E_{22}-E_{11}](\X)\nonumber 
\end{align}
in the notation \eqref{foggy}.
We call $(N,k)$ the {\em type} of the PCF variety
$V(\B)_{N,k}$.
In the special case that $\alpha\in\OO$ but $\beta=\sqrt{\alpha}\notin \OO$,
we shorten the notation  to $V(\alpha)_{N,k}:=V(\{\beta,-\beta\})_{N,k}$.
\end{definition1}

\begin{proposition}
\label{grinder}
\begin{enumerate}
\item\label{grinder1}
The variety $V_{N,k}\subseteq \A^{N,k}$ does not depend
on $N$ in the following sense: Let 
\[
\pi_{N,k}:\A^{N,k}=\A^N\times \A^k\rightarrow \A^k
\]
be projection onto the second factor.  Then $V_{N,k}=\pi_{N,k}^{-1}(V_{0,k})$.
\item\label{grinder2}
If $k=1$, then $V_{N,k}=\emptyset$.
\item\label{grinder3}
We have
$\dim V_{N,k}=\begin{cases} -1 &\text{ if }k=1\\
                          N &\text{ if }k=2\\
                          N+k-3 &\text{ if }k\geq 3.
\end{cases}$
\end{enumerate}
\end{proposition}
\begin{proof}
The variety $V_{N,k}\subset \A^{N,k}$ is defined by requiring that 
\[
M[y_1, \ldots, y_N]M([x_1, \ldots , x_k])
M[y_1, \ldots, y_N]^{-1}
\]
equal a multiple of the identity as in Definition \ref{muskdeer}.
This is true if and only if $M([x_1,\ldots , x_k])$
is a multiple of the identity, proving \ref{grinder1}.

To prove \ref{grinder2}, it suffices by \ref{grinder1} to remark that 
$V_{0,1}=\emptyset$ since $M([x_1])=D(x_1)$ is never a multiple
of the identity.

To prove \ref{grinder3}, observe that\\[.014in]
\hspace*{.3in}$\dim V_{0,k}= 
\begin{cases} -1 &\text{ if } k=1\\
               0 &\text{ if } k=2\\
              k-3 &\text{ if } k\geq 3
\end{cases}$\qquad  and \qquad 
$\dim V_{N,k}=\begin{cases} -1 &\text{ if } k=1\\
                               N+\dim V_{0,k} &\text{ if } k>1.
\end{cases}$
\end{proof}

\begin{remark1}
Because of Proposition \ref{grinder}\ref{grinder1}, we can
eliminate the variables $y_1,\ldots, y_N$ in \eqref{toady}
and write the equations 
defining $V_{N,k}\subseteq \A^{N,k}$ as 
follows: Let $E_{0,k}=\begin{bmatrix} \E_{11} & \E_{12}\\
\E_{21} & \E_{22}\end{bmatrix}$. Then we have
\begin{equation}
\label{deadly}
V_{N,k}:\E_{12}(x_1, \ldots, x_k)=\E_{21}(x_1, \ldots , x_k)=
[\E_{22}-\E_{11}](x_1, \ldots, x_k)=0.
\end{equation}
\end{remark1}

\subsection{First Properties of PCF Varieties}

\begin{enumerate}
  \item
The dimension of $V_{N,k}$ is $N+k-3$ if $k\geq 3$ by Proposition 
\ref{grinder}\ref{grinder3}.
The dimension of $V(\B)_{N,k}$ is generically 
$N+k-2$.\footnote{We know of three degenerate cases where the dimension of
a component is $> N+k-2$. These are $(N,k)=(1,2)$ with $A=B=0$
in Section \ref{sec:integral}, $(N,k)=(0,1)$ with $C=-A$ in
\ref{gerd}, and $(N,k)=(0,2)$ with $A=B=0$ or $B=C=0$ in \ref{truck}.}
So generically $V_{N,k}$ is a divisor on $V(\B)_{N,k}$ for
any $\B$ if $k\geq 3$.
If $P=\PP$ is an $\OO$-PCF with value
$\ob(P)\in \B$, then
\[
p=p(P)=\p\in V(\B)_{N,k}(\OO)\subseteq \A^{N,k}(\OO).
\]
\item
For integers $\ell\ge0$, $m\ge1$ the map
\[
(b_1,\ldots,b_N, a_1,\ldots,  a_{k})\mapsto (d_{1},\ldots,d_{N+\ell}, 
c_1,\ldots,   c_{mk})
\]
defined by \eqref{onto} induces inclusions
$V(\B)_{N,k}\hookrightarrow V(\B)_{N+\ell, mk}$
of PCF varieties by Proposition \ref{ratify}.
  \item
Conjugates of PCF varieties are PCF varieties:
For $\sigma\in\Gal(\overline{\Q}/\Q)$, let
$\sigma(\B)$ be the multi-set $\{\sigma(\beta),\sigma(\bs)\}$.
Then $V(\B)_{N,k}^{\sigma}=V(\sigma(\B))_{N,k}$.
\item
Let $P$ be a PCF of type $(N,k)$ with roots
$\B=B(P)$.
The dual construction $P\mapsto P^\ast$ of Galois \cite{galois} induces
a linear involution on $V(\B)_{N,k}\subseteq \mathbb{A}^{N,k}$ if $N>0$.
It induces an inclusion
$V(\B)_{0,k}\hookrightarrow V(\B)_{1,k}$ if $N=0$.
\end{enumerate}

\subsection{PCF Varieties \texorpdfstring{\protect{\boldmath{$V(\B)_{N,k}$}}}{} fiber over Fermat-Pell conics}

Suppose now that
$\B=\{\beta,\bs\}$ is the multi-set of  roots of $Ax^2+Bx+C\in\OO[x]$,
where we assume throughout this subsection that $A\neq 0$.
Since $Ax^2+Bx+C=A(x-\beta)(x-\bs)$ we have 
$A(x\beta +y)(x\bs+y)=Cx^2-Bxy+Ay^2$.
\begin{definition1}
Suppose $A\neq 0$. The {\em Fermat-Pell conic} $\FP_k(\B)_{/\OO}$ is the
plane curve
\begin{equation}
\label{Fermat}
\FP_k(\B) : A(x\beta+y)(x\bs +y) =Cx^2-Bxy+Ay^2=(-1)^k A,
\end{equation}
which is irreducible if $B^2-4AC\neq 0$.
If $\alpha\in\OO$ and $\beta=\sqrt{\alpha}\notin\OO$, we
write $\FP(\alpha)$ for $\FP(\{\beta, -\beta\})$.  In this case
we recover the familiar Fermat-Pell equation, namely
\[
\FP_{k}(\alpha): y^2-\alpha x^2=(-1)^k.
\]
\end{definition1}
\begin{theorem}
\label{FPfibration}
There is a fibration $\pi_{\FP}:V(\B)_{N,k}\rightarrow \FP_k(\B)$
over the Fermat-Pell conic defined by
\begin{equation}
\label{FPmap}
\pi_{\FP}(b_1, \dots,b_N,a_1, \ldots , a_{k})=(E_{21}(b_1,\ldots, b_N, 
a_1, \cdots , a_{k}),
E_{22}(b_1,\ldots ,b_N, a_1, \ldots,  a_{k}))
\end{equation}
with the polynomials $E_{21}$, $E_{22}$ as in Definition \textup{\ref{muskdeer}}
and Proposition \textup{\ref{rounded}}.
\end{theorem}
\begin{proof}
This follows from Proposition \ref{returned}
applied to $E=E_{N,k}(b_1,\ldots ,b_N, a_1, \ldots ,   a_{k})$
 and Remark \ref{rindss}.
\end{proof}

\begin{definition1}
For $A\neq 0$, let $\Sigma_{k}\subseteq \FP_{k}(\B)$ 
be the zero-dimensional subscheme
\begin{equation*}
\Sigma_k=
\langle x=0,\, y^2=(-1)^k\rangle.
\end{equation*}
\end{definition1}

\begin{proposition}
For $A\neq 0$ and the Fermat-Pell fibration
 $\pi_{\FP}:V(\B)_{N,k}\rightarrow \FP_k(\B)$, we have
$\pi_{\FP}^{-1}(\Sigma_{k})=V_{N,k}\subseteq V(\B)_{N,k}$.
\end{proposition}
\begin{proof}
Assume $\B=\{\beta,\bs\}$ with
$\OO[x]\ni A(x-\beta)(x-\bs)=Ax^2+Bx+C=0$,
$A\neq 0$.
Then we have that $p\in V(\B)_{N,k}$ is in 
$\pi_{\FP}^{-1}(\Sigma_{k})\subseteq V(\B)_{N,k}$ if and only if 
$E_{21}(p)=0$. But since
$\Quad(p)=E_{21}x^2+(E_{22}-E_{11})x-E_{12}$ must be a multiple of
$Ax^2+Bx+C$ with $A\neq 0$, this multiple must be $0$, i.e.,
$E$ is a multiple of the identity.
But $V_{N,k}$ was defined by the equations 
specifying that $E(p)$ was a multiple of the identity.
\end{proof}
\noindent If $(0,\lambda)$ is
on $\FP_{k}(\B)$, then $\lambda =\pm i^k$ and 
\begin{equation*}
\pi_{\FP}^{-1}(0,\lambda)=\{p\in V(\B)_{N,k}\mid E(p)=\lambda 
\Id_{2\times 2}\}.
\end{equation*}

\subsection{The 
\texorpdfstring{\protect{\boldmath{$\Z_2$}}}{Z\textsubscript{2}}-Extension of
\texorpdfstring{\protect{\boldmath{$\Q$}}}{Q}}

The prototypical PCF, essentially
known to the ancient Greeks, is $\sqrt{2}=[1,\overline{2}]$.
We are interested in generalizing this ur-example to
the $\Z_2$-extension of $\Q$.
We review this tower of number fields now and set notation.
For an integer $m\ge 1$, let $\zeta_{m}$ be the 
primitive $m$th root of unity $e^{2\pi i/m}$ with
\begin{equation*}
\alpha_n=\zeta_{2^{n+2}}+ \overline{\zeta}_{2^{n+2}}=2\cos(2\pi/2^{n+2})
\end{equation*}
for $n\ge -1$.
Hence, $\alpha_{-1}=-2$, $\alpha_0=0$, $\alpha_1=\sqrt{2}$,
$\alpha_2=\sqrt{2+\sqrt{2}}$, $\alpha_3=\sqrt{2+\sqrt{2+\sqrt{2}}}$, etc.
The relation
\begin{equation*}
\alpha_n^2=2 + \alpha_{n-1}
\end{equation*}
holds for $n\ge 0$.
The totally real number field $F_n:=\Q(\alpha_n)=\Q(\zeta_{2^{n+2}})^+$
is Galois over $F_0=\Q$ with $\Gal(F_n/\Q)\cong \Z/2^n \Z$ and
$F_{n+1}$ a quadratic extension of $F_n$ for $n\ge 0$.
Furthermore, $F_\infty = \bigcup_{n\ge 0}F_n$
is the (unique) $\Z_2$-extension of $\Q$.  The integers of
the number field $F_n$ are $\OO_n:=\Z[\alpha_n]$ for $n\geq 0$.

With this notation $\alpha_{1}=\sqrt{2}=[1,\overline{2}]$
gives the point $(1,2)\in V(2)_{1,1}(\Z)$. The problem of finding
{\em all\/} $\Z$-PCFs for $\sqrt{2}$ of type
$(1,1)$ would involve finding the integral points on $V(2)_{1,1}$.
We will see in Proposition \ref{wind} that
\[
V(2)_{1,1}(\Z)=\{\pm (1,2)\}\subseteq \A^{1,1}(\Z)
\]
with $[1,\overline{2}]$ converging to $\sqrt{2}$ and
$[1,-\overline{2}]$ converging to $-\sqrt{2}$.
Hence there is only one $\Z$-PCF  of $\sqrt{2}$ of type
$(1,1)$.
The generalization
of this to the $\Z_2$-extension of $\Q$ is:
\begin{problem}
Find the $\Z[\alpha_n]$-PCFs for $\alpha_{n+1}$,
$n\geq 0$.
\end{problem}
\noindent The associated diophantine problem
for PCF varieties is:
\begin{problem}
Find 
$V(2+\alpha_{n})_{N,k}(\OO_n)$, $n\geq 0$.
\end{problem}

\section{Convergence}
\label{convergence}

Let  $P=\PP$ be an $\OO$-PCF with roots
$\B=\B_{N,k}(P)$; put $V:=V(\B)_{N,k}$.
Say that a point $p\in V(\C)$ is convergent if
the PCF $P(p)$ converges and divergent
otherwise.  For any $\OO$-algebra $R\subseteq \C$ set
\begin{align}
V(R)^{\textrm{con}} & = \{p\in V(R)\mid p\mbox{ is convergent}\}\label{don}\\
V(R)^{\textrm{div}} &=\{p\in V(R)\mid p\mbox{ is divergent}\}\nonumber.
\end{align}
We then have
\begin{equation}
\label{err}
V(R)=V(R)^{\textrm{con}} \coprod V(R)^{\textrm{div}}.
\end{equation}
We want to understand the decomposition \eqref{err};
Corollary \ref{uber} will show that 
\[
V_{N,k}(R)\subseteq V(R)^{\textrm{div}}.
\]

The convergence of $P=[b_1,\ldots, b_N,\overline{a_1, \ldots , a_k}]$
 was understood
in the nineteenth century (see, e.g., \cite{j-t}*{Chapter 3}).
It only depends on the $a_i$, but the criteria for convergence
involve both algebraic conditions and inequalities.
To give the inequalities, make the $a_i$ periodic as in \eqref{periodic}.
Then we have the condition for divergence:
\begin{ineq}
\label{grits}
$M([a_{j+1},\ldots, a_{k+j}])_{21}=0$ and 
$\left|M([a_{j+1}, \ldots, a_{k+j}])_{22}\right|>1$
for some $0\leq j\leq k-1$.
\end{ineq}
In this section we give the practical
Algorithm \ref{con} which given $P$ answers
the questions
\begin{enumerate}
\item Does $P$ converge?
\item If so, which element of $\B(P)$ is its value (or limit)?
\end{enumerate}
We defer proofs to the appendix.

\begin{definition1}
Let $A=\left[\begin{smallmatrix}a&b\\c&d\end{smallmatrix}\right]$ with $ad-bc=\varepsilon=\pm1$.
Let $\lambda_\pm$ be the eigenvalues of $A$ chosen so that
$|\lambda_+|\ge1\ge|\lambda_-|$.
If $A\ne\pm\sqrt{\varepsilon} I$ let
\[\beta_\pm=\beta_\pm(A)=\frac{\lambda_\pm-d}c=\frac{b}{\lambda_\pm-a}
\left(=\frac{a-\lambda_\mp}c=\frac{b}{d-\lambda_\mp}\right)\in\BP^1(\C)
\]
where we take whichever expression is not the indeterminate $0/0$.
\end{definition1}

\begin{theorem}\label{bactrian}
Let $P=\PP$ be
a PCF.
Then the value $\ob(P)$ exists if and only if \textbf{none} of the following
three conditions is satisfied:
\begin{enumerate}
\item\label{bactrian1}
$E(P)=\pm i^k I$.
\item\label{bactrian2}
\textup{INEQ \ref{grits}} holds.
\item\label{bactrian3}
$\Tr(E(P))^2\in\R$ and $0\le(-1)^k\Tr(E(P))^2<4$.
\end{enumerate}
If it converges, then the value $\ob(P)=\beta_+(E(P))$.
\end{theorem}

\begin{corollary}
\label{uber}
Set $V:=V(\B)_{N,k}$.  Then $V_{N,k}(\C)\subseteq V(\C)^{\textrm{div}}$ for any
$\B$.
\end{corollary}

\begin{proof}
We have $p\in V_{N,k}(\C)$ if and only if $E(p)$ is a multiple of $I$.
\end{proof}

We reformulate Theorem \ref{bactrian} as the following algorithm.
Step 2 corresponds to checking for Case \ref{bactrian1}, and
Step 6 corresponds to checking for Case \ref{bactrian3}.
\begin{algorithm}
\label{con}
\SetKw{Continue}{continue}
\SetKw{Break}{break}
\SetKwInOut{Input}{Input}\SetKwInOut{Output}{Output}
\Input{A periodic continued fraction $P=\PP$.}
\Output{Either its value (or limit) $\ob(P)$ or ``doesn't exist''.}
compute $E=E(P)$ from \eqref{fed1}\;
\lIf{$E$ is a multiple of the identity}{print ``doesn't exist''; \Break}
compute the multi-set of roots $\B(P)=\{\beta,\bs\}$ of 
$E_{21}(P)x^2+[E_{22}-E_{11}](P)x-E_{12}(P)$\;
\lIf{$\beta=\bs$}{$\ob(P):=\beta=\bs$; return $\ob(P)$; \Break}
WLOG assume $\beta\ne\infty$ and compute $\left|E_{21}\beta+E_{22}\right|$\;
\lIf{$\left|E_{21}\beta+E_{22}\right|=1$}{print ``doesn't exist''; \Break}
\leIf{$\left|E_{21}\beta+E_{22}\right|>1$}{$\ob(P):=\beta$}{$\ob(P):=\bs$}
\leIf{\textup{INEQ \ref{grits}}  is satisfied}{print ``doesn't exist''; \Break;}
{return $\ob(P)$}
\caption{ The value (or limit) of a periodic continued fraction}\label{pudu5}
\end{algorithm}

\begin{remark1}
\label{pudu}
  \begin{enumerate}
    \item[]
\item The conditions are ordered from the least likely to the most likely to
occur, with convergence (none of them occurring) the most likely of all;
cf.~Proposition \ref{giraffe}.
\item  The dual PCF
\[P^*=[b_1 , b_2,\ldots, b_{N},0, \overline{-a_{k},-a_{k-1},\ldots,-a_1}]\]
converges under the same conditions, except the inequality is reversed in the
second condition.
If it converges, the limit is the ``other'' fixed point $\beta_+(E(P^*))=\beta_-(E(P))$, cf. Corollary \ref{hungry}.
\item\label{pudu3}
Since $(E_{21}z+E_{22})(E_{21}z-E_{11})=-(-1)^k$, we could use the inequality
\[|E_{21}\beta_+(E)-E_{11}|<1<|E_{21}\beta_-(E)-E_{11}|\] instead in the algorithm.
Furthermore, $|E_{21}z+E_{22}|$ tells us how fast the continued fraction
converges: each successive convergent provides 
\begin{equation*}
\frac{2}{k}\log_{10}(\left|E_{21}z+E_{22}\right|)=-\frac{2}{k}\log_{10}
(\left|E_{21}z-E_{11}\right|)
\end{equation*}
decimal digits of accuracy on average.\footnote{ A simple calculation shows
that the $(nk)$-th convergent has error $O(1/\lambda_1^{2n})$ where
$\lambda_1=E_{21}\beta_+(E)+E_{22}$ is the larger eigenvalue of $E$.}
If $z=\beta_+(E)=\beta_-(E)$, and therefore the limit exists and $|E_{21}z+E_{22}|=1$,
exponentially many convergents are needed for each additional digit
of accuracy.
For example the $m$-th convergent of $[\overline{2i}]$ is $i+i/m$, so the
$n$-th digit isn't accurate until past the $10^n$-th convergent.
\item As an example of what can go wrong if \textup{INEQ \ref{grits}} is
satisified, the convergents $c_n$ of $[\overline{2,-1/2,1}]$ have the following
behavior:
\[\lim_{n\to\infty}c_{3n}=5=\lim_{n\to\infty}c_{3n+1},\text{ but }c_{3n+2}=0\text{ for every }n.\]




\end{enumerate}
\end{remark1}

\section{PCF Varieties of Type 
\texorpdfstring{$(N,k)$}{(N,k)} when 
\texorpdfstring{$N+k\le 2$}{N+k<3}}
\label{drive}

Since the dimension of a PCF variety of type $(N,k)$ is
$N+k-2$ and $k\ge1$, there are three types where PCF varieties
consist of a finite set of points.

\subsection{Type \texorpdfstring{\protect{\boldmath{$(0,1)$}}}{(0,1)}}
\label{gerd}

We have 
\[
E_{0,1}=E=\begin{bmatrix} E_{11} & E_{12}\\E_{21} & E_{22}\end{bmatrix}
=\begin{bmatrix} x_1 & 1\\1 & 0\end{bmatrix}
\]
from \eqref{foggy}.
Hence if $\B=\{\beta,\bs\}$ are roots of the quadratic polynomial $Ax^2+Bx+C$,
the variety $V(\B)_{0,1}$ is given \eqref{foggy1} by
\begin{equation}
\label{route}
A(-x_1)  =  B,\quad -A(1)  =  C,\quad -B(1)  =  C(-x_1).
\end{equation}
From \eqref{toady}, the variety $V_{0,1}=\emptyset$.
Nominally the dimension of $V(\B)_{0,1}$ is $-1$,
which means that only special $\beta$, $\bs$  give a nonempty variety.
For these special $\beta$ the variety has exactly one point.
The condition for $V(\B)_{0,1}$
to be nonempty from \eqref{route} is $C=-A\ne0$ or,
equivalently, $\beta\bs=-1$.
The one-point variety consists of $a_1=-B/A$, which is in $V(\B)_{0,1}(\OO)$
if and only
if $\beta+\bs\in\OO$.
In case $V(\B)_{0,1}\neq\emptyset$, the Fermat-Pell conic $\FP_{1}(\B)$ 
of \eqref{Fermat} 
is given by
$-x^2+(\beta + \bs)xy+y^2=-1$, and the map $\pi_{\FP}:V(\B)_{0,1}\rightarrow \FP_{1}(\B)$
by \eqref{FPmap}
is $\pi_{\FP}(a_1)=(1,0)$.

We have thus proved:

\begin{proposition}

If $V(\B)_{0,1}\neq \emptyset$ and $\beta+\bs=0$,
then $\beta=\pm 1$.
In particular, the variety $V(2+\alpha_n)_{0,1}=\emptyset$,
and there are no $\Z[\alpha_n]$-PCFs  for $\alpha_{n+1}$
of type $(0,1)$ for $n\ge 0$.
\end{proposition}

According to Algorithm \ref{pudu5} let
$\beta,\bs$ be the two roots of $x^2-a_1x-1=0$.
If $a_1=\pm 2i$ the value of the PCF  $P=[\overline{a_1}]$
exists and $\ob(P)=\beta=\bs=\pm i$.
If $\beta\ne\bs$ but $|\beta|=1$, then the value does not exist.
If $|\beta|>1>|\bs|$, then $\ob(P)=\beta$.
It is easy enough to say what happens more directly.
If $a_1^2\in\R$ and $-4<a_1^2\le 0$, then the value does not exist.
Otherwise, $\ob(P)=\beta_+(E(P))$.

\subsection{Type \protect{\boldmath{$(0,2)$}}}
\label{truck}

From \eqref{foggy}, 
\[
E_{0,2}=E=\begin{bmatrix} E_{11} & E_{12}\\E_{21} & E_{22}\end{bmatrix}
=\begin{bmatrix} x_1x_2+1 & x_1\\x_2 & 1\end{bmatrix}.
\]
Hence if $\B=\{\beta,\bs\}$ are roots of $Ax^2+Bx+C$,
the variety $V(\B)_{0,2}$ is given \eqref{foggy1} by
\begin{equation*}
A(-x_1x_2)  =  B(x_2),\quad -A(x_1)  =  C(x_2),\quad -B(x_1)  =  C(-x_1x_2).
\end{equation*}
From \eqref{toady}, the variety $V_{0,2}$ is given by
$x_2=x_1=-x_1x_2=0$, so $V_{0,2}\subseteq V(\B)_{0,2}(\C)^{\textrm{div}}$ 
in the notation of \eqref{don}  is the point
$(0,0)$.
If $A,B,C\ne0$, then $V(\B)_{0,2}$ consists of two points $(a_1,a_2)=(-B/A,B/C)$
and the  point $(0,0)$ constituting the divisor 
$V_{0,2}\subseteq V(\B)_{0,2}$.
If $B=0\ne A,C$; $C=0\ne B$; or $A=0\ne B$,
then there is only the extraneous point $(0,0)$.
If $B=C=0$, then $V(\{0,0\})_{0,2}$ is one-dimensional,
consisting of the line $x_1=0$,
and if $A=B=0$, then $V(\{\infty,\infty\})_{0,2}$ is also one-dimensional,
consisting of the line $x_2=0$.
The Fermat-Pell conic $\FP_{2}(\B)$ is given \eqref{Fermat}
by $Cx^2-Bxy+Ay^2=A$ and the map $\pi_{\FP}:V(\B)_{0,2}\rightarrow
\FP_2(\B)$ by $\pi_{\FP}(x_1,x_2)=(x_2,1)$.
We have thus proved:
\begin{proposition}
\label{wind}
\begin{enumerate}
  \item
For $n\ge1$, $V(\alpha_{n-1})_{0,2}$ consists of
the extraneous point $(a_1,a_2)=(0,0)$.
\item  There are no 
$\Z[\alpha_n]$-PCFs of type $(0,2)$
for $\alpha_{n+1}$, $n\ge 0$.
\end{enumerate}
\end{proposition}

According to Algorithm \ref{pudu5} let
$\beta,\bs$ be the two roots of $a_2x^2-a_1a_2x-a_1=0$.
If $(a_1,a_2)=(0,0)$, then the value $\ob(P)$  of the PCF
$P=[\overline{a_1,a_2}]$ does not exist.
If $a_1=0\ne a_2$, then $\ob(P)=\beta=\bs=0$.
If $a_1\ne0=a_2$, then $\ob(P)=\beta=\bs=\infty$.
If $a_1a_2=-4$, then $\ob(P)=\beta=\bs=a_1/2=-2/a_2$.
If $\beta\ne\bs$ but $|a_1\beta+1|=1$, then $\ob(P)$ does not exist.
If $|a_2\beta+1|>1>|a_2\bs+1|$, then $\ob(P)=\beta$.

\subsection{Type \texorpdfstring{\protect{\boldmath{$(1,1)$}}}{(1,1)}}

From \eqref{foggy}, 
\[
E_{1,1}=E=\begin{bmatrix} E_{11} & E_{12}\\E_{21} & E_{22}\end{bmatrix}
=\begin{bmatrix} y_1& x_1y_1+1-y_1^2\\1 & x_1-y_1\end{bmatrix}.
\]
Hence if $\B=\{\beta,\bs\}$ are roots of $Ax^2+Bx+C$,
the variety $V(\B)_{1,1}\subseteq \A^{1,1}$ is given \eqref{foggy1} by
\begin{equation*}
A(x_1-2y_1)  =  B,\quad A(y_1^2-x_1y_1-1)  =  C,\quad 
B(y_1^2-x_1y_1-1)  =  C(x_1-2y_1).
\end{equation*}
From \eqref{toady}, the variety $V_{1,1}$ is given by
$1=x_1y_1+1-y_1^2=x_1-2y_1=0$, so $V_{1,1}$ is the empty set.
If $A=0$, $V(\B)_{1,1}$ is empty.
Otherwise, the variety consists of two points (counting multiplicity)
$(y_1, x_1)=(b_1, a_1)$
in a potentially quadratic extension with
\[a_1=\pm\sqrt{\frac{B^2-4AC}{A^2}-4}\quad \text{ and }\quad b_1=\frac12\left(-\frac{B}A+a_1\right).\]
The Fermat-Pell conic $\FP_{1}(\B)$ is given \eqref{Fermat}
by $Cx^2-Bxy+Ay^2=-A$ and the map $\pi_{\FP}:V(\B)_{1,1}\rightarrow
\FP_1(\B)$ by $\pi_{\FP}(y_1, x_1)=(1, x_1-y_1)$.

Since $(x-\alpha_{n+1})(x+\alpha_{n+1})=x^2-(2+\alpha_{n})$, we have the
following proposition.
\begin{proposition}
\begin{enumerate}
\item  The variety $V(2+\alpha_n)_{1,1}$ consists of two
points $(b_1,a_1)$ with
\begin{equation*}
a_1=\pm 2\sqrt{1+\alpha_{n}}\quad\mbox{and}\quad b_1=a_1/2.
\end{equation*}
It has no points rational over $\Q(\alpha_{n})$ if $n\geq1$.
The PCF variety  $V(2)_{1,1}$ consists of
the two points $(b_1,a_1)=\pm(1,2)$.
\item  There are no $\Z[\alpha_{n}]$-PCFs
of type $(1,1)$ for $\alpha_{n+1}$ if $n\geq 1$.
If $n=0$, there is precisely one: $\sqrt{2}=[1,\overline{2}]$.
\end{enumerate}
\end{proposition}
\begin{proof}

We need only remark that
$1-\alpha_{n}<0$ if $n\geq 1$, and hence its Galois conjugate
$1+\alpha_{n}$ cannot be a square in $\Q(\alpha_{n})\subset\R$.
\end{proof}

According to Algorithm \ref{pudu5} let
$\beta,\bs$ be the two roots of $x^2+(a_1-2b_1)x-b_1a_1+b_1^2-1=0$.
If $a_1=\pm2i$, the value of the PCF $P=[b_1,\overline{a_1}]$
exists and $\ob(P)=\beta=\bs=b_1\mp i$.
If $\beta\ne\bs$ but $|\beta-b_1|=1$, then the value does not exist.
If $|\beta-b_1|<1<|\bs-b_1|$, then $\ob(P)=\beta$.

\section{PCF Curves of Type \texorpdfstring{$(0,3)$}{(0.3)}} 
\label{toady1}

There are three types -- $(N,k)=(0,3), (2,1), (1,2)$ --  where the
PCF variety $V(\B)_{N,k}$ is $1$-dimensional.  The remainder of this paper
is devoted to analyzing these PCF curves, beginning in this section
with type $(0,3)$.

From \eqref{foggy}
we have
\begin{equation*}
E_{0,3}= E= D(\ox_1)D(\ox_2)D(\ox_3)=
\begin{bmatrix}\ox_1\ox_2\ox_3+\ox_3+\ox_1 & \ox_1\ox_2+1\\
\ox_2\ox_3+1 & \ox_2\end{bmatrix}.
\end{equation*}
Hence the variety $V_{0,3}\subseteq \A^{0,3}$ is defined by
\begin{equation*}
x_1x_2+1=x_2x_3+1=x_2-x_1x_2x_3-x_3-x_1=0,
\end{equation*}
and so $V_{0,3}$ consists of the two points $\pm(i,i,i)$.
If $\B=\{\beta,\bs\}$ are the roots of
$Ax^2+Bx+C$, then from \eqref{foggy1}
the variety $V(\B)_{0,3}\subseteq \A^{0,3}$ is given by
\begin{align}
A(\ox_2-\ox_1\ox_2\ox_3-\ox_1-\ox_3) & = B(\ox_2\ox_3+1),\nonumber\\
A(-\ox_1\ox_2-1) & = C(\ox_2\ox_3+1),\label{gift}\\
B(-\ox_1\ox_2-1) & = C(\ox_2-\ox_1\ox_2\ox_3-\ox_1-\ox_3).\nonumber
\end{align}
Eliminating $\ox_1$ we can rewrite the Equations \eqref{gift} as the plane curve
\begin{equation}
\label{riot}
V(\B)_{0,3}\colon A(\ox_2^2+1)-B\ox_2(\ox_2\ox_3+1)+C(\ox_2\ox_3+1)^2=0.
\end{equation}
Hence the curve $V(\B)_{0,3}$ has genus 0 and, as we shall see, has a rational
point if and only if $i\in K$ or $B^2-4AC$ is a nonzero sum of squares.

\begin{proposition}
For even $k$, the variety $V(\B)_{N,k}$ always has
a $K$-rational point.
For odd $k$, if the variety $V(\B)_{N,k}$ has a $K$-rational
point, then $B^2-4AC$ is the sum of two squares in $K$.
\end{proposition}

\begin{proof}

We showed in \ref{truck}  that $V(\B)_{0,2}$
always has a $K$-rational point, so we are done in the even $k$ case.
If $\infty\in\B$, then $A=0$ so $B^2-4AC=B^2$
is a square in $K$.
Otherwise, let $P$ be the PCF associated to the
$K$-rational point.
Note that, up to a square, $B^2-4AC$ is equal to
\[[E_{22}-E_{11}]^2(P)+ 4[E_{12}E_{21}](P)=\Tr(E(P))^2-4(-1)^k,\]
which is the sum of two squares if $k$ is odd.
\end{proof}

If $B^2-4AC=R^2+S^2$ with $R,S\in K$, we can obtain an explicit
parametrization of the rational curve $V(\B)_{0,3}$ over $K$.

\begin{proposition}
\label{para}
Suppose $B^2-4AC=R^2+S^2\ne0$ with $R,S\in K$, and $C\ne-A$.
Then all points $(a_1,a_2,a_3)$ on $V(\B)_{0,3}$ over $K$ are given by
\begin{align*}
a_1&=\frac{-2C(t^2+1)+S(t^2-1)+2Rt}{B(t^2+1)-R(t^2-1)+2St},\\
a_2&=\frac{B(t^2+1)-R(t^2-1)+2St}{-S(t^2-1)-2Rt},\\
a_3&=\frac{2A(t^2+1)+S(t^2-1)+2Rt}{B(t^2+1)-R(t^2-1)+2St}
\end{align*}
for some $t\in\BP^1(K)$.\qed
\end{proposition}

For example if $\beta=\alpha_1=\sqrt{2}$ we get the parametrization
\[
(a_1,a_2,a_3)=\left(-\frac{t^2+2}{t^2-2},-\frac{t^2-2}{t^2-4t+2},
-\frac{2t}{t^2-2}\right)
\]
by setting $A=1$, $B=0$, $C=-2$, $R=2$, $S=-2$ and adjusting $t\to t-1$.
Four integral solutions are $(a_1,a_2,a_3)=\pm(3,-1,2)$ ($t=1,2$)
                       and $\pm(1,1,0)$ ($t=0,\infty$).
The conditions $B^2\ne4AC$ and $C\ne-A$ are necessary to make the curve
geometrically irreducible.
When $C=-A\ne0$, the parametrization above,
$(a_1,a_2,a_3)=(t,-(2At+B)/(At^2+Bt-A),t)$, only gives one component.
The other component is given by $a_2=0, a_3=-B/A-a_1$.
When $B^2-4AC=0$ both geometric components are parametrized above,
but one with $R=Si\ne0$ and the other with $R=-Si\ne0$.

Let
 $\OO\subset\C$ be the ring of $S$-integers in a number field $K$
and suppose that $\B=\{\beta, \bs\}$ is the roots of 
$Ax^2+Bx+C\in\OO[x]$.
\begin{theorem}
\label{great}
If $AC(A+C)(B^2-4AC)\ne0$, then
$V(\B)_{0,3}(\OO)$ is finite.
Hence any $\beta\in\overline{\Q}\setminus\OO$ with $\beta\bs\ne-1$ has
only finitely many $\OO$-PCFs of type $(0,3)$.
\end{theorem}

\begin{proof}
If $AC(A+C)(B^2-4AC)\ne0$ then the curve has four geometric points at
infinity on the normalization using the parametrization in Proposition
\ref{para}.
Hence, by Siegel's Theorem~\ref{mousedeer} below, 
the curve has finitely many $\OO$-points.
\end{proof}

\begin{theorem}[Siegel {\cite{siegel}, \cite[p.~95]{serre-mw}}]\label{mousedeer}
If an affine rational curve over $\OO$ has three or more geometric points at
infinity after normalization, then it has finitely many $\OO$-points.
An affine curve of positive genus always has finitely many
$\OO$-points.
\end{theorem}

\begin{remark1}
In the excluded case where $AC(A+C)(B^2-4AC)=0$ it is possible to
get infinitely many $\OO$-points.
For example, let $A=1$, $B=-1$, $C=0$, and $\OO=\Z[\sqrt{2}]$. Then
\[(x_1,x_2,x_3)=((1+\sqrt{2})^n,-(-1+\sqrt{2})^n,-(1+\sqrt{2})^{2n}+(1+\sqrt{2})^n-1) \]
for $n\in\Z$ gives infinitely many $\Z[\sqrt{2}]$-points on $V(\{1,0\})_{0,3}$ corresponding
to infinitely many $\Z[\sqrt{2}]$-PCFs
$[\overline{x_1,x_2,x_3}]$ of type $(0,3)$ with value~$0$.
\end{remark1}

The Fermat-Pell conic $\FP_{3}(\B)$ is given \eqref{Fermat}
by $Cx^2-Bxy+Ay^2=-A$ and the map $\pi_{\FP}:V(\B)_{0,3}\rightarrow
\FP_3(\B)$ by $\pi_{\FP}(x_1, x_2, x_3)=(x_2x_3+1, x_2)$ from \eqref{FPmap}.

\subsection{\texorpdfstring{\protect{\boldmath{$\Z$}}}{Z}-Points on 
\texorpdfstring{\protect{\boldmath{$V(2)_{0,3}$}}}{V(2)(0,3)}}

For $\beta=\sqrt{2}=\alpha_1$, we have  $A=1$, $B=0$, and $C=-2$.
So Equations \eqref{gift} say $V(2)_{0,3}$ is given by
\begin{align*}
\ox_2-\ox_1\ox_2\ox_3-\ox_1-\ox_3 & = 0\\
 \nonumber -\ox_1\ox_2-1 & = -2(\ox_2\ox_3+1),
\end{align*}
so in particular
\begin{equation*}
\ox_2(\ox_1-2\ox_3)=1.
\end{equation*}
For $\ox_1,\,\ox_2, \,\ox_3\in\Z$ we must therefore have $\ox_2=\pm 1$.
Hence we have:
\begin{proposition}
\begin{enumerate}
\item  There are precisely four $\Z$-points on
$V(2)_{0,3}$.
They are 
\[
(\ox_1,\ox_2,\ox_3) = \pm (1,1,0), \, \pm(3,-1,2).
\]
\item  There are exactly two $\Z$-PCFs
for $\sqrt{2}$ of type $(0,3)$. They are $\sqrt{2}=[\overline{3,-1,2}]$
and $\sqrt{2}=[\overline{1,1,0}]$,   which is equivalent to the classical
$\sqrt{2}=[1, \overline{2}]$ by Lemma \textup{\ref{caribou}}.
\end{enumerate}
\end{proposition}
\noindent Use Algorithm \ref{con} to verify that the limits exist and
are as stated.

\subsection{\texorpdfstring{\protect{\boldmath{$\Z[\sqrt{2}]$}}}
{Z[sqrt(2)]}-Points on 
\texorpdfstring{\protect{\boldmath{$V(2+\sqrt{2})_{0,3}$}}}{V(2+sqrt(2))(0,3)}}
\label{sec22}

For $\beta=\alpha_{2}$, we have $A=1$, $B=0$, and
\begin{equation}
\label{rinsed}
C=-(2+\alpha_1)=-uw, \mbox{ where }
w=\sqrt{2}, \,\, u=1+w.
\end{equation}
Equation \eqref{riot} for $V(2+\sqrt{2})_{0,3}$ is 
\begin{equation}
\label{pots}
(2+\sqrt{2})(1+\ox_2\ox_3)^2-\ox_2^2=1,
\end{equation}
with $\ox_1$ determined by $\ox_1\ox_2+1=(2+w)(\ox_2\ox_3+1)$.

Setting $b=\ox_2$ and $a=\ox_3$ we can rewrite \eqref{pots} as 
\begin{equation}
\label{rattle}
uw(ab+1)^2=b^2+1,\quad\mbox{or}\quad b(uwa^2b+2uwa-b)=-u .
\end{equation}
Hence $b$ must be a unit in the ring $\Z[\sqrt{2}]$, so
\begin{equation*}
b=\pm u^n
\end{equation*}
for some $n\in\Z$ and $\|b\|=\pm 1$, where
for any $t=r+s\sqrt{2}\in \Q(\sqrt{2})$ we define 
\[
\|t\|=\N_{\Q(\sqrt{2})/\Q}(t)=r^2-2s^2.
\]

First suppose $\|b\|=1$.  Then $b=m+2nw$ with
$m, n\in\Z$, $m$ odd, implying that  $\| b^2 +1\|\equiv 4\bmod8$.
But \eqref{rattle} then gives
\begin{equation*}
  2\|ab+1\|^2=\| b^2+1\|\equiv 4 \pmod8 .
\end{equation*}
However, $4 \bmod8$ is not twice a square, so $\|b\|=1$ is impossible,
and we must have $\| b\|=-1$. If $(a,b)$ is a solution to \eqref{rattle}, so is
$-(a,b)$.  So, without loss of generality, suppose $b=u^{2j+1} = m+nw$ with $j\in\Z$,
so $b+b^{-1}=2nw$.  Then
\begin{equation}
  \label{video}
  \frac{b}{uw}(b+b^{-1})=\frac{u^{2j+1}}{uw}(2nw)=u^{2j}2n .
\end{equation}
However, by \eqref{rattle} we have
\begin{equation*}
  (ab+1)^2=\frac{1}{uw}(b^2+1)=\frac{b}{uw}(b+b^{-1}) .
\end{equation*}
Hence \eqref{video} shows that $u^{2j}2n$ must be a square in $\Z[w]$, which means
that the integer $n$ is either a square or twice a square.
The latter is impossible since
\[m+nw=b=u(u^2)^j=(1+w)(3+2w)^j\equiv(1+w)\bmod2.\]
Hence $n=y^2$ and
\[\|b\|=(m+y^2w)(m-y^2w)=m^2-2y^4=-1.\]
We thus recover Ljunggren's equation \cite{ljunggren}  with the complete
solution set over $\Z$ given by $(m,y)=(\pm 1, \pm 1)$ and $(m,y)=
(\pm 239, \pm 13)$. This corresponds to $b=x_2=u^{\pm 7}$ and $b=u^{\pm 1}$.
Hence we have
\begin{theorem}
  \label{trails}

  There are precisely sixteen $\Z[\sqrt{2}]$-points on the
  curve $V(2+\sqrt{2})_{0,3}$.  With $w=\sqrt{2}$
  and $u=1+w$, they are
  \begin{align*}
    \pm(a_1, a_2, a_3) =
    & (3-w, 1+w=u, 3-2w), (1+w, -1-w=-u, 1),\\
    & (1+w, -1+w=u^{-1}, -1), (3+3w, 1-w=-u^{-1}, 1+2w),\\
    & (57-39w, 239+169w=u^7, -73+52w) ,\\
    &(-421+299w, -239-169w=-u^7, -551+390w) ,\\
    & (203+143w, -239+169w=u^{-7}, -109-78w) ,\\
    & (681+481w, 239-169w=-u^{-7}, 369+260w) .
  \end{align*}
  \end{theorem}
  \begin{corollary}
    \label{rinds}

    There are precisely eight $\Z[\sqrt{2}]$-PCFs of
    $\alpha_2=\sqrt{2+\sqrt{2}}$  of type $(0,3)$.
    They are 
    \begin{align*}
\sqrt{2+\sqrt{2}}
     & = [\overline{3-\sqrt{2},\,1+\sqrt{2},\, 3-2\sqrt{2}}],\\
     & = [\overline{1+\sqrt{2},\,-1-\sqrt{2},\,1}],\\
     & = [\overline{1+\sqrt{2},\,-1+\sqrt{2},\, -1}],\\
     & = [\overline{3+3\sqrt{2},\,1-\sqrt{2},\, 1+2\sqrt{2}}],\\
     & = [\overline{57-39\sqrt{2},\,239+169\sqrt{2},\, -73+52\sqrt{2}}],\\
     & = [\overline{-421+299\sqrt{2},\,-239-169\sqrt{2},\, -551+390\sqrt{2}}],\\
     & = [\overline{203+143\sqrt{2},\,-239+169\sqrt{2},\, -109-78\sqrt{2}}],\\
     & = [\overline{681+481\sqrt{2},\,239-169\sqrt{2},\, 369+260\sqrt{2}}].
    \end{align*}
    \end{corollary}

\begin{proof}

Again we check convergence using Algorithm \ref{con}.
\end{proof}

The last two PCFs converge extremely slowly.
For these $[\overline{a_1, a_2, a_3}]$, we have
$|(a_1a_2+1)\alpha_2+a_1|\approx 1.002094$,
which means every extra significant digit requires 
$\approx 3/(2\log_{10}(1.002094))\approx 1651$
more convergents by Remark \ref{pudu}\ref{pudu3}.

\section{PCF curves of type \texorpdfstring{$(2,1)$}{(2,1)}}
\label{ghost}

We now consider PCF curves of type $(2,1)$.
From \eqref{foggy} we have
\begin{align*}
E=E_{2,1}&= D(y_1)D(y_2)D(x_1)D(0)D(-y_2)D(-y_1)D(0)\\
& =\begin{bmatrix}-y_1y_2^2-y_2+y_1y_2x_1+y_1+x_1 & y_1^2y_2^2+2y_1y_2+1
-y_1^2y_2x_1-y_1x_1
-y_1^2\\
-y_2^2+y_2x_1+1 & y_1y_2^2+y_2-y_1y_2x_1-y_1\end{bmatrix}.\nonumber
\end{align*}
The variety $V_{2,1}$ from \eqref{deadly} is given
by $1=-x_1=1=0$, so it is empty as we know
from Proposition \ref{grinder}\ref{grinder2}.

Let $\B=\{\beta,\bs\}$ be the roots of $Ax^2+Bx+C\in\OO[x]$.  Then
from Definition \ref{pint}
the variety $V(\B)_{2,1}\subseteq \A^{2,1}$ is given by
\begin{align}
A(2y_1y_2^2-2y_1y_2x_1-2y_1+2y_2-x_1)&=B(y_2x_1-y_2^2+1),\label{chorusline}\\
A(-y_1^2y_2^2+y_1^2y_2x_1+y_1^2-2y_1y_2+y_1x_1-1)&=C(y_2x_1-y_2^2+1),\nonumber\\
B(-y_1^2y_2^2+y_1^2y_2x_1+y_1^2-2y_1y_2+y_1x_1-1)&\nonumber\\
=C(2y_1y_2^2-2y_1y_2 &x_1-2y_1+2y_2-x_1). \nonumber 
\end{align}
Generically the curve \eqref{chorusline} has genus $1$ and is given by
\begin{align}
y^2&=(B^2-4AC)(y_2^2+1)^2-4A^2\text{ or}\nonumber\\
z^2&=-4(Ay_1^2+By_1+C)^2+B^2-4AC\text{ where}\nonumber\\
y&=(2Ay_1+B)(y_2^2+1)+2Ay_2=-Ax_1-\frac{B^2-4AC}{A}y_2(y_2^2-y_2x_1-1)\text{ and}\nonumber\\
z&=x_1(Ay_1^2+By_1+C).\text{ Rearranging, }B^2-4AC\text{ is a sum of squares}\nonumber\\
  B^2-4AC&=\left(2Ay_1+B+\frac{2Ay_2}{y_2^2+1}\right)^2+\left(\frac{2A}{y_2^2+1}\right)^2\label{dreamers}\\
         &=(x_1^2+4)(Ay_1^2+By_1+C)^2.\nonumber\\
         &=\frac{A^2(x_1^2+4)}{(y_2^2-y_2x_1-1)^2},\text{ noting that}\nonumber\\
         A&=(Ay_1^2+By_1+C)(y_2^2-y_2x_1-1).\nonumber
  \end{align}
The $j$-invariant of its jacobian is
\[j=\frac{256(B^2-4AC-3A^2)^3}{A^4(B^2-4AC-4A^2)}.\]

If, say, $i\in K$ then the curve \eqref{dreamers} has rational points
\[(y_1,y_2,x_1,y)=\left(-\frac{B}{2A}+\frac{i}{2},i,
\frac{2i(B^2-4AC-A^2)}{B^2-4AC+A^2},2iA\right), (\infty,i,2i,-2iA).\]

\begin{theorem}
\label{rentdue}

If $B^2-4AC\ne0$ or $i\not\in K$, then
the curve $V(\B)_{2,1}(\OO)$ is finite.
Hence, any $\beta\in\overline{\Q}\setminus\OO$ has only finitely many $\OO$-PCFs
of type $(2,1)$.
\end{theorem}

\begin{proof}

The curve in Equation \eqref{dreamers} has genus 1 if and only if
$A(B^2-4AC)(B^2-4AC-4A^2)\ne0$, in which case
we are done by Siegel's Theorem~\ref{mousedeer}.
Otherwise, the curve has genus 0, which we now analyze case by case.
If $A=0\ne B$ we can let $y_2$ be the parameter and
\[(y_1,y_2,x_1,y)=\left(-\frac{C}B-\frac{y_2}{y_2^2+1},y_2,
  \frac{y_2^2-1}{y_2},B(y_2^2+1)\right),\]
so again we are done by Siegel since there are four points at infinity
\linebreak $x_2=\pm i,0,\infty$.
If $A=B=0$ we have the parametrization $(y_1,y_2,x_1,y)=(y_1,\pm i,\pm2i,0)$,
but this corresponds to $\beta=\bs=\infty$.
If $B^2-4AC-4A^2=0\ne A$ we have the parametrization
\begin{align*}
(y_1,y_2,x_1,y)=\bigg(-&\frac{B}{2A}-\frac{t^2-6t+1}{t^2-2t+5},\frac{t^2-6t+1}{2(t^2-1)},\\
&\frac{(t^2+2t-7)(t^2-6t+1)}{4(t-3)(t^2-1)},
-\frac{(3t^2-2t+3)(t^2-6t+1)A}{2(t^2-1)^2}\bigg),
\end{align*}
so again we are done by Siegel since 
$t=\pm 1,3,1\pm2i,\infty$ gives six points at infinity.
If $B^2-4AC=0$ then \eqref{dreamers} says $y^2+4A^2=0$, so if
$i\not\in K$ then $A=0$, a case we have already covered.
Finally note that if $\beta\not\in\OO$ then $\bs$ is uniquely determined
and $B^2-4AC\ne0$.
\end{proof}

\begin{remark1}
Suppose $B^2-4AC=0$ and $i\in K$.
Firstly if $B^2-4AC=0$, then $\beta=\bs=-B/(2A)$ and $V(\B)_{2,1}$
has genus $0$.
A parametrization is
\[(y_1,y_2,x_1,y)=\left(-\frac{B}{2A}-\frac1{y_2-i},\,y_2,\,2i,\,-2iA\right),\]
which is integral if \{$\beta=-B/(2A),i,y_2\}\subseteq\OO$, 
and $y_2-i$ is a unit.
Hence we can have infinitely many integral points in this case.
\end{remark1}

The Fermat-Pell conic $\FP_{1}(\B)$ is given by
$Cx^2-Bxy+Ay^2=-A$ from \eqref{Fermat} and the map
$\pi_{\FP}:\A^{2,1}\supseteq V(\B)_{2,1}\rightarrow \FP_{1}(\B)$ is given by
\[
\pi_{\FP}(y_1,y_2, x_1) =
(-y_2^2+y_2x_1+1 , y_1y_2^2+y_2-y_1y_2x_1-y_1).
\]

\begin{proposition}
The PCF curve $V(2)_{2,1}$ has exactly four 
$\Z$-points:
\begin{equation*}
\nonumber  (y_1,y_2,x_1,y)=(1,0,-2,2),(-1,-2,-2,2),(1,2,2,-2),(-1,0,2,-2).
\end{equation*}
Hence, $\sqrt{2}$ has exactly two $\Z$-PCFs
of type $(2,1)$:
\[
 \sqrt{2}=[-1,0,\overline{2}]\quad\mbox{and}\quad \sqrt{2}=[1,2,\overline{2}].
\]
\end{proposition}
\begin{proof}

The quartic $-4(y_1^2-2)^2+8=4(-y_1^4+4y_1^2-2)$ from Equation~\eqref{dreamers}
only has only two integral $y_1$ giving positive $z^2$ values, let alone square
integer values, namely $\pm1$,
which yield
\[(y_1,y_2,x_1,y)=(1,0,-2,2),(-1,-2,-2,2),(1,2,2,-2),(-1,0,2,-2).\]
\end{proof}
\begin{proposition}
  \label{grunt}
  \begin{enumerate}
\item  
We have $V(2+\sqrt{2})_{2,1}(\Z[\sqrt{2}])=\emptyset$.
\item  
There are no $\Z[\sqrt{2}]$-PCFs
  of $\sqrt{2+\sqrt{2}}$ of type~$(2,1)$.
\end{enumerate}
\end{proposition}

\begin{proof}
  By \eqref{dreamers}, the PCF curve $V(2+\sqrt{2})_{2,1}$
  is given by (with $w=\sqrt{2}$ as in Section \ref{sec22})
  \begin{equation}
    \label{pounds}
    z^2=-4(y_1^2-2-w)^2+8+4w=-4y_1^4+8(2+w)y_1^2-4(4+3w)
  \end{equation}
with $w=\sqrt{2}$ as in \eqref{rinsed}.
  For \eqref{pounds} to have a solution in $\Z[w]$, we must have
  that 
\[
-y_1^4+(4+2w)y_1^2-(4+3w)
\]
 is a square in $\Z[w]$.  But this is
  impossible $\bmod\, 4$ : $-y_1^4+(4+2w)y_1^2-(4+3w)$ must be $w$ or
  $-1-w$ $\bmod\, 4$, whereas the squares in $\Z[w]$
  are $0$, $1$, $2$, or $3+2w$ $\bmod\, 4$.
  \end{proof}

\section{PCF curves of type \texorpdfstring{$(1,2)$}{(1,2)}}
\label{sec:integral}

Lastly we consider the PCF curves with type $(1,2)$.
These are substantially more difficult than the types $(0,3)$ and
$(2,1)$ considered in Sections \ref{toady1} and \ref{ghost}.

From \eqref{foggy}
\begin{equation*}
E=E_{1,2}= D(y_1)D(x_1)D(x_2)D(0)D(-y_1)D(0)
=\begin{bmatrix}y_1x_1+1& y_1x_1x_2 +x_2-y_1^2x_1\\
x_1 & x_1x_2-y_1x_1+1\end{bmatrix}.
\end{equation*}
The variety $V_{1,2}$ from \eqref{deadly} is given
by $x_1=x_2= -x_1x_2=0$, so it is a line;
we know from Proposition \ref{grinder}\ref{grinder3} that
$\dim V_{1,2}=1$.
From Definition \ref{pint}
the curve  $V(\B)_{1,2}$ is given by
\begin{align*}
A(x_1x_2-2y_1x_1)&= Bx_1, \\
A(y_1^2x_1-y_1x_1x_2-x_2)&=Cx_1,\\
B(y_1^2x_1-y_1x_1x_2-x_2)&=C(x_1x_2-2y_1x_1). 
\end{align*}

The affine curve $V(\B)_{1,2}$ has the line $V_{1,2}: x_1=x_2=0$
as one component.
This is the only component if $A=0\ne B$.
If $A\ne0$ a second component $V(\B)_{1,2}'$ is the rational
curve given by
\begin{equation}
\label{redo}
(y_1,x_1,x_2)=\left(y_1,-\frac{2Ay_1+B}{Ay_1^2+By_1+C},\frac{2Ay_1+B}A\right).
\end{equation}
If $B^2=4AC,A\ne0$ a third component is $(y_1,x_1,x_2)=(-B/(2A),x_1,0)$.
Finally, if $A=B=0$ there is only one component---the
two-dimensional one with $x_1=0$, which contains the first
component mentioned above.
If $x_2\ne0$ in this case the PCF has the value $\infty$.
For the remainder of the section we shall assume $A\ne0$ and focus on
$V(\B)_{1,2}'$.

\begin{theorem}
  \label{tadpole}
  \begin{enumerate}
\item  If $B^2\ne4AC$, then $V(\B)_{1,2}'(\OO)$
 is finite.
\item  An algebraic number $\beta\in\overline{\Q}\setminus\OO$ has only finitely
many $\OO$-PCFs of type $(1,2)$.
\end{enumerate}
\end{theorem}

\begin{proof}
By Equation \eqref{redo}, there are three points at infinity, $y_1=\infty$ or
a root of $Ay_1^2+By_1+C$, so we are done by Siegel's Theorem \ref{mousedeer}.
\end{proof}

Rearranging Equation \eqref{redo}, we obtain
\begin{equation}
\label{frog1}
  V(\B)_{1,2}'\colon Ay_1^2x_1+By_1x_1+B +2Ay_1=-Cx_1.
\end{equation}
For the special case of $\alpha\in\OO$, $\sqrt{\alpha}\notin \OO$, 
\eqref{frog1} simplifies to
\begin{equation*}
  V(\alpha)_{1,2}'\colon y_1^2x_1+2y_1=\beta^2y_1.
\end{equation*}
In particular
\begin{equation}
  \label{goods1}
  V(2+\alpha_n)_{1,2}' \colon y_1^2x_1+2y_1=(2+\alpha_{n})x_1.
\end{equation}
The $\OO_{n}$-points on $ V(2+\alpha_n)_{1,2}'$
come in natural ``families'' of size $2^{n+1}$, which is clearest to explain by
a general lemma:
\begin{lemma}
  
  \label{sandy}
Let $\OO\subset\C$ be a Dedekind domain.
Let $m\in\OO$, and let
$\sigma$ be an automorphism of $\OO$ such that $\sigma(m)/m=u^2$ for
some unit $u\in\OO^\times$.
If $(y_1,x_1)\in\OO^2$ is a solution to $y_1^2x_1+2y_1=mx_1$,
then so is $(u^{-1}\sigma(y_1),u\sigma(x_1))$.
\end{lemma}

We need the following lemma so that we can take $m=\alpha_{n+1}^2$.
(Recall $F_{n+1}=\Q(\alpha_{n+1})$ from the introduction.)

\begin{lemma}
\label{shinto}
Let $\sigma$ generate $\Gal(F_{n+1}/\Q)$.
Then $\sigma^j(\alpha_{n+1})/\alpha_{n+1}$ is a unit in $F_{n}$.
\end{lemma}

\begin{proof}

Since $\N_{F_{n+1}/\Q}(\alpha_{n+1})=2$ and $2$ ramifies completely in $F_{n+1}$,
$\sigma^j(\alpha_{n+1})$ equals $\alpha_{n+1}$ up to a unit in $F_{n+1}$.
Also $F_{n}$ is the fixed field of the involution of $F_{n+1}$ that takes
$\alpha_{n+1}$ to $-\alpha_{n+1}$, which 
clearly fixes $\sigma^j(\alpha_{n+1})/\alpha_{n+1}$.
\end{proof}
 Applying Lemmas \ref{sandy} and \ref{shinto} we have the following
\begin{proposition}
  \label{runts}
Let $\sigma$ generate $\Gal(F_{n+1}/\Q)$ and let 
$u_j=\sigma^j(\alpha_{n+1})/\alpha_{n+1}$
  for $0\le j< 2^{n+1}$; the $u_j$ are distinct units in $F_{n}$
with $u_{2^{n}}=-1$. If $(y_1,x_1)$ is
an $\OO_{n}$-point on $ V(2+\alpha_{n})_{1,2}'$ in
\textup{(\ref{goods1})}, then so
  is $(u_j^{-1}\sigma^j(y_1), u_j\sigma^j(x_1))$.
\end{proposition}
 A consequence of Proposition \ref{runts} is the following:
\begin{proposition}
  \label{stunts}

The number $N_n$ of $\OO_{n}$-points on $ V(2+\alpha_{n})_{1,2}'$
  satisfies $N_n\equiv\,1\pmod{2^{n+1}}$.
The number of $\OO_{n}$-PCFs
  of $\alpha_{n+1}$ of type $(1,2)$ is divisible by $2^{n}$.
\end{proposition}
\begin{proof}

  The $\OO_{n}$-points on $ V(2+\alpha_{n})_{1,2}'$ consist of
  the extraneous point $(y_1,x_1)=(0,0)$ together with families
  $(u_j^{-1}\sigma^j(y_1), u_j\sigma^j(x_1))$ with $y_1\ne 0\ne x_1$,
  which correspond to $\OO_{n}$-PCFs for
  $\alpha_{n+1}$ or $-\alpha_{n+1}$ of type $(1,2)$.

To see that the $2^{n+1}$ members of a family are distinct, suppose
$x_1=u_j\sigma^j(x_1)$ for some $j\ne0$.
Then $\sigma^j(\alpha_{n+1}x_1)=\alpha_{n+1}x_1$ and because some power of 
$\sigma^j$
equals $\sigma^{2^{n}}$, $-\alpha_{n+1}x_1
=\sigma^{2^{n}}(\alpha_{n+1}x_1)=\alpha_{n+1}x_1$.
So $x_1=0$, which implies $y_1=0$, a contradiction.

Note that Equation \eqref{goods1} implies
$x_1x_2(x_1x_2+4)=4y_1x_1(y_1x_1+2)=4x_1^2\alpha_{n+1}^2>0$, which implies
$x_1x_2<-4$ or $x_1x_2>0$, so by Algorithm \ref{con}
the PCFs from the families all converge.
  \end{proof}

We next reduce finding integral points on $ V(2+\alpha_{n})_{1,2}'$
to finding integral points on another curve.  We give a general lemma.

\begin{lemma}
  
  \label{radios}
Let $\OO\subset\C$ be a Dedekind domain,
let $p$ be a rational prime that has a unique
$\OO$-prime $\mathfrak p$ above it, and assume $\mathfrak p=(\pi)$ is principal.
If $y_1,x_1\in\OO$ satisfy $y_1^2x_1+py_1=\pi x_1$,
then $p\mid x_1$ and $(x_1/p)\mid y_1$,
so that $a=x_1/p, b=py_1/x_1\in\OO$ satisfy $a^2 b^2+b=\pi$.
Conversely, if $a,b\in\OO$ satisfy $a^2 b^2+b=\pi$,
then $x_1=pa, y_1=ab\in\OO$ satisfy $y_1^2x_1+py_1=\pi x_1$.
\end{lemma}

\begin{proof}

We can write $p=\pi^eu$ for some unit $u$ and $e>0$.
We have $\pi\mid y_1^2x_1=\pi x_1-\pi^e uy_1$ so $\pi\mid x_1$ or $\pi\mid y_1$.
If $\pi\mid y_1$ we get $\pi(y_1/\pi)^2x_1+\pi^eu(y_1/\pi)=x_1$, which implies
$\pi\mid x_1$, so $\pi\mid x_1$ in either case.
We have $y_1^2(x_1/\pi)+\pi^{e-1}uy_1=\pi(x_1/\pi)$, which is identical to the
original equation with $x_1/\pi$ replacing $x_1$ and $e-1$ replacing $e$,
so we can repeat this process $e$ times to show that $p=\pi^eu\mid x_1$.
Now we have $y_1=\pi a-y_1^2a=a(\pi-y_1^2)$ with $a=x_1/p$, so $a\mid y_1$.
Letting $y_1=ab$ and dividing by $a$ we obtain $b=\pi-(ab)^2$.
\end{proof}

\begin{remark1}
\label{pinto}
Applying Lemma \ref{radios} with $\OO=\OO_{n}=\Z[\alpha_{n}]$, $p=2$,
and $\pi=2+\alpha_{n}$, we see that $\OO_{n}$-points on
$ V(2+\alpha_{n})_{1,2}'$ correspond one-to-one with $\OO_{n}$-points
on the curve
\begin{equation}
  \label{oranges}
  E(2+\alpha_{n}) \colon (a^2 b+1)b=2+\alpha_{n} .
\end{equation}
  \end{remark1}

The curve $E(2+\alpha_{n})$ is in turn closely related
to the curve
\[
F(2+\alpha_{n})=
\BP^1\setminus \{ \infty, \alpha_{n+1}, -\alpha_{n+1}\}
\]
given over $\OO_{n}$ by the equation
\begin{equation}
  \label{turret}
  u[t^2-(2+\alpha_{n})] = 1 .
\end{equation}
\begin{proposition}

  The curves $E(2+\alpha_{n})$ and $F(2+\alpha_{n})$
  are isomorphic over $\OO_{n}[1/2]$.
\end{proposition}
\begin{proof}

  An explicit isomorphism from $E(2+\alpha_{n})$ in (\ref{oranges})
  to $F(2+\alpha_{n})$ in (\ref{turret}) is given by
  $t=ab$, $u=-1/b$ with inverse $b=-1/u$, $a=-ut$.
\end{proof}
 One could also deduce the finiteness of the $\OO_{n}$-points
on $E(2+\alpha_{n})$ via the finiteness
of the $\OO_{n}[1/2]$-points on $\BP^1\setminus\{0,
\alpha_{n+1}, -\alpha_{n+1}\}$, which is a common application of Siegel's theorem.
In the following we shall find all solutions to $(a^2 b+1)b=2+\alpha_{n}$,
then recover the PCF as $\alpha_{n+1}=[ab,\overline{2a,2ab}]$.

\subsection{\texorpdfstring{\protect{\boldmath{$\Z$}}}{Z}-points on the Curve
  \texorpdfstring{\protect{\boldmath{$E(2)$}}}{E(2)}}

\begin{proposition}

The $\Z$-points on the curve
\begin{equation*}
E(2)\colon 2=(a^2b+1)b
\end{equation*}
are $(a,b)=(1,1)$, $(-1,1)$, $(-1,-2)$, $(1,-2)$, $(0,2)$.
\end{proposition}
\begin{proof}

Firstly, $b\mid 2$ implies $b=\pm1,\pm2$, and then we can solve for $a$ in each case.
\end{proof}
As a consequence we deduce the following continued fractions.
\begin{corollary}

There are precisely
two $\Z$-PCFs for $\sqrt{2}$ of type $(1,2)$:
\begin{equation*}
\sqrt{2}=[1,\overline{2,2}]\qquad\mbox{and}\qquad\sqrt{2}=[2,\overline{-2,4}].
\end{equation*}
\end{corollary}

\subsection{\texorpdfstring{\protect{\boldmath{$\Z[\sqrt{2}]$}}}{Z[sqrt(2)]}-points on the curve
\texorpdfstring{\protect{\boldmath{$E(2+\sqrt{2})$}}}{E(2+sqrt(2))}}

\begin{theorem}
\label{tips}

Let $w=\sqrt{2}$ and $u=\sqrt{2}+1$.
Then the $\Z[\sqrt{2}]$-points on the curve
\begin{equation}
\label{round}
E(2+\sqrt{2})\colon 2+\sqrt{2}=uw=(a^2b+1)b
\end{equation}
are $(a,b)=(0,wu)$ and the ten pairs
$(\pm a, b)$ with $a,\, b$ as follows:
  \begin{align*}
  b&=-u,\, a=-1, \\
  b&=u,\,a=w-1=1/u, \\
  b&=1/u^3,\, a=13+9w=(3+w)u^2, \\
  b&=-u^5,\, a=31-22w=-(3-w)/u^3, \\ \noalign{\medskip}
  b&=w,\,a=1, \\
  b&=-wu^2,\, a=1-w=-1/u, \\ \noalign{\medskip}
  b&=-w/u, a=-2-w=-wu, \\
  b&=-wu^3, a=4-3w=-w/u^2\\
  b&=-w/u^{13}, a =-13wu^{10}, \\
  b&=-wu^{15}, a=-13w/u^{11}.
  \end{align*}
\end{theorem}

 Most of the rest of this section is devoted to proving
Theorem~\ref{tips}.

By Equation \eqref{round}, $b|w$ in $F_1$, so $b=\pm u^k,\pm wu^k$.
We divide the proof up into four cases: $\|b\|=1,-1,-2,2$, which correspond to
$b=\pm u^{2k}$, $\pm u^{2k+1}$, $\pm wu^{2k}$, $\pm wu^{2k+1}$.

We change variables to simplify.
Let $x=-wb/u\in \Z[w]$, so $\|x\|=2\|b\|$.
The identity $(ab)^2=uw+uwx/2=uw(x+2)/2$ implies
\begin{align}
\|ab\|^2&=\|uw\|\|x+2\|/4=(x+2)\left(\frac{\|b\|}{x}+1\right)\nonumber\\
\text{and so }y^2&=x(x+2)(x+\|b\|)\text{ where }y=x\|ab\|,\label{kouprey}
\end{align}
which is elliptic if $\|b\|\ne2$.
Hence, it is necessary and sufficient to find all the
$\Z[w]$-points $(x,y)$ on the curve \eqref{kouprey}
with $w(x+2)/(ux^2)=a^2$ a square in $\Z[w]$ and $\|x\|=2\|b\|$.

We pause in this proof to prove a theorem that follows from the ``method of
descent'' of Fermat and from an application of this method by Bessy.

\begin{theorem}
\label{saiga}

The elliptic curve $E\colon y^2=x^3-x$ has Mordell--Weil group
$E(F_1)\cong\Z/4\Z\times\Z/2\Z$.
The seven affine $F_1$-rational points are
\[(x,y)=(0,0),(1,0),(-1,0),(1-w,\pm(2-w)),(1+w,\pm(2+w)).\]
\end{theorem}

\begin{proof}
 Let $\sigma$ generate $\Gal(F_1/\Q)$, and denote the
quadratic twist of\linebreak $E$ by 2 by $E^\sigma\colon y^2=x^3-4x$,
which is $F_1$-isomorphic to $E$ but not $\Q$-isomorphic.  Fermat
showed in effect that
$E(\Q)=\langle(0,0),(1,0)\rangle=E[2]\cong(\Z/2\Z)^2$, and
Bessy\footnote{ Dickson \cite[Ch.~XXII, page 617]{dickson2} 
     attributes this to \cite{frenicle} Bernard 
  Fr\'{e}nicle de Bessy (c.~1604--1674 [Dickson's ``\dag 1765'' must
    be a transposition typo for 1675]), page 175 of a posthumous
  ``Trait\'{e} des Triangles Rectangles en Nombres, Paris, 1676,
  101--6; M\'{e}m.\ Acad.\ Sc.\ Paris, 5, 1666--1699; \'{e}d.\ Paris
  5, 1729, 174; Recu[e]il de plusieurs traitez [sic] de
  math\'{e}matiques de l'$\!$Acad.\ Roy.\ Sc.\ Paris,
  1676''\kern-.07em.  Bessy corresponded regularly with Fermat
  (1607--1665).  } showed in effect that
$E^\sigma(\Q)=\langle(0,0),(2,0)\rangle=E^\sigma[2]\cong(\Z/2\Z)^2$.

Suppose $P\in E(F_1)$.
Then
\begin{equation*}
2P = (P + P^\sigma) + (P - P^\sigma).
\end{equation*}
Now $(P + P^\sigma)^\sigma = P + P^\sigma$, so $P + P^\sigma \in E(F_1)^\sigma=E(\Q)$.
Likewise $(P - P^\sigma)^\sigma = -(P - P^\sigma)$, which identifies
$P - P^\sigma$ with a point on $E^\sigma(\Q)$.
So we can recover $E(F_1)$ as the preimage in $E(F_1)$ of
$E(\Q) + E^\sigma(\Q)$ under multiplication by~$2$.
Note that the sum $E(\Q) + E^\sigma(\Q)$ need not be direct, but
$E(\Q) \cap E^\sigma(\Q)$ is contained in the \hbox{$2$-torsion} subgroup
$E[2] = E^\sigma[2]$.
Since $E(\Q)=E[2]$ and $E^\sigma(\Q)=E^\sigma[2]$, we have $E(F_1)\subset E[4]$.
The 4-division polynomial of $E$ is $(x^3-x)(x^2+1)(x^4-6x^2+1)$, which has
seven roots in $F_1$.  Five of these give $F_1$-rational $y$'s,
which give the seven points in the theorem.
Alternatively, we could have used Magma's \texttt{Generators} \cite{magma} command.
\end{proof}

 \ul{Case $\|b\| = 1$.}  This easiest case is just a matter
of considering\linebreak Equation~\eqref{round} mod $4$. Indeed
$\|b\| = 1$ implies $b = m + n w$ with $(m,n) \equiv (1,0) \pmod 2$.  But
then $2+w - b = (2-m) + (1-n)w$ has norm $(2-m)^2 - 2(1-n)^2 \equiv 3
\pmod 4$, which is not a square in~$\Z$, so $2+w - b=(ab)^2$ cannot be
a square in $\Z[w]$.

Alternatively, Equation \eqref{kouprey} with $\|b\|=1$ is the elliptic curve
$E\colon y^2=(x+1)^3-(x+1)$, which is $\Q$-isomorphic to the $E$ in
Theorem \ref{saiga}.
So $E$ has only affine points
\[(x,y)=(0,0),(-1,0),(-2,0),(-w,\pm(2-w)),(w,\pm(2+w)),\]
which have $\|x\|=0,1,4,-2$, none of which equals $2\|b\|=2$.

\ul{Case $\|b\| = -2$.}
Equation \eqref{kouprey} with $\|b\|=-2$ is the elliptic curve
$E\colon y^2=x^3-4x$, which is $F_1$-isomorphic to the $E$ in
Theorem \ref{saiga}.
So $E$ has only affine points
\[(x,y)=(0,0),(2,0),(-2,0),(2-2w,\pm(4w-4)),(2+2w,\pm(4w+4)),\]
the last four of which have $\|x\|=-4=2\|b\|$ and $w(x+2)/(ux^2)$ a square.
These correspond to the points $(a,b)=(\pm(1-w),-4-3w=-wu^2),(\pm1,w)$,
which correspond to the 5th and 6th pairs in the theorem.

\ul{Case $\|b\| = 2$.}
In this case the curve \eqref{kouprey} is rational, and the norm of $x$ needs
to be $4$.
One solution is $(x,y)=(-2,0)$ which gives the extraneous solution
$(a,b)=(0,uw)$.
Otherwise, to get integral points we let $t=y/(x+2)$ be integral, so that
$(x,y)=(t^2,t(t^2+2))$ with $\|t\|=\pm2$.
Hence, we can take $t=\pm wu^j$, $x=2u^{2j}$, and $b=-wu^{2j+1}$.

\begin{lemma}
Let $(z_1,z_2,z_3)$ be a $\Z[w]$-solution to \eqref{gift} with
$(A,B,C)=(1,0,-uw)$.
Then $(\pm a,b)=\left((z_2z_3+1)/ z_2^2 ,-uwz_2^2\right)$ are
$\Z[w]$-solutions to \eqref{round} with $\|b\|=2\|z_2\|^2=2(-1)^2=2$
by the proof of Theorem \textup{\ref{trails}}.
Conversely, if $(a,b)$ is a $\Z[w]$-solution to \eqref{round} with
$\|b\|=2$ then $(z_1,z_2,z_3)=\left(z_2(1\pm ab)(a^2b+1),z_2,z_2(a^2b\pm a+1)\right)$,
where $z_2^2=-b/(uw)$, is a $\Z[w]$-solution to \eqref{gift}.
\end{lemma}

\begin{proof}
Since $z_2$ is a unit of norm 1 and $-b/(wu)$ is the square of a unit from
above, the proof reduces to algebraic verification that the two maps are
well defined and inverse to each other.
\end{proof}

The pair $(\pm a,b)$ corresponds to a quadruplet of $(z_1,z_2,z_3)$,
so the sixteen solutions in Theorem \ref{trails} correspond to four pairs
of solutions here, namely $b=-wu^{2j+1}$ with $j=\pm1,\pm7$ and
$a=\pm(4-3w),\pm(2+w),\pm(149266-105547w),\pm(61828+43719w)$,
respectively.
In terms of $y_1,x_1,x_2$ the first map is
\[\left(y_1,x_1,x_2)=\pm(z_1z_2+1,-\frac{2}{z_2^2}(z_2z_3+1),2(z_1z_2+1)\right).\]

\ul{Case $\|b\| = -1$.}
We try to proceed as in the first two cases.
The curve $E\colon y^2=x(x+2)(x-1)$, elliptic curve 96A1 
in \cite{cremona} and 96.b3 in \cite{lmfdb}, has 
$E(\Q)=E[2]$.  The curve
$E^\sigma\colon y^2=x(x+4)(x-2)$, elliptic curve 192A2 in \cite{cremona}
and 192.a2 in \cite{lmfdb}, has
$E^\sigma(\Q)\cong\Z\times\Z/2\Z\times\Z/2\Z$.
The point $(x,y)=(4,8)$ is a generator for $E^\sigma(\Q)/\text{torsion}$, which corresponds to
$(x,y)=(2,-2w)$ on the original curve $E$.
Looking for rational preimages of $\langle(2,-2w),(1,0),(0,0)\rangle$ under
$[2]$ we find an additional point $P:=(w,w)$ satisfying $2P=(2,-2w)$.
So $E(F_1)$ is
$\langle P,(1,0),(0,0)\rangle\cong\Z\times\Z/2\Z\times\Z/2\Z$,
which Magma's \texttt{Generators} tells us directly.

Searching $E(F_1)=nP+E[2]$, $n\in\Z$, for $n<100$, the largest $n$ we
find with coordinates in $\Z[w]$ is $n=4$.  So
$E$ appears to have 23 affine $\Z[w]$-points:
\begin{align*}
(x,y)=&(0,0),(-2,0),(1,0),(-1,\pm w),(2,\pm2w),(4,\pm6w),\\
&(w,\pm w),(-w,\pm w),\left(-3w + 4,\pm(9w - 12)\right),\left(3w + 4,\pm(9w + 12)\right),\\
&\left(17w + 24,\pm(119w + 168)\right),\left(-17w + 24,\pm(119w - 168)\right),(25,\pm90w),\end{align*}
with $\|x\|=0,4,1,1,4,16,-2,-2,-2,-2,-2,-2,625$.
Of the six $x$'s with norm $-2$ only the 1st, 2nd, 5th, and 6th have
square $w(x+2)/(ux^2)$, giving
\begin{align*}
(a,b)=&(\pm1,-1-w),\left(\pm(1-w),1+w\right),\\
&\left(\pm(13+9w),-7+5w\right),\left(\pm(31-22w),-41-29w\right),
\end{align*}
which correspond to the first four pairs in the theorem.
We now have
to roll up our sleeves and do some more work to show we get no other solutions
in this case.

To simplify notation, for the remainder of  this section we set
$K=F_1=\Q(\alpha_1)=\Q(\sqrt{2})$.
Write $b = \pm u z^2$ for some unit $z$ of $\Z[w]$.
Setting $a':=ab$, we have ${a'}^2 = wu - b = wu \mp u z^2$ and
\be
\label{eq:pell_u}
wu = {a'}^2 \pm u z^2,
\ee
which is a ``generalized (Fermat--)Pell equation'' over $\Z[w]$
subject to the additional condition that $z$ be a unit.

The two choices of sign are equivalent under Galois conjugation,
because the conjugate of $uz^2$ is $-u^{-1} \ov{z}^2 = -u(\ov{z}/u)^2$.
We choose the minus sign, so that \eqref{eq:pell_u} is equivalent to
\begin{equation}
\label{eq:pell_v}
wu = \N_{L_1/K} (a' + vz),
\end{equation}
where
\begin{equation}
\label{roast}
v=u^{1/2}
\end{equation}
and $L_1$ is the quadratic extension $K(v)$ of~$K$.
Because $u$ has one positive and one negative conjugate,
this extension $L_1$ is a quartic number field with two real embeddings
and one conjugate pair of complex embeddings.  By Dirichlet's unit theorem,
then, the group $U_{L_1}$ of units of~$L_1$ has rank~$2$.  The image of
the norm map $\N_{L_1/K} : U_{L_1} \ra U_K$ has rank~$1$, because it
contains the unit $u^2 = \N_{L_1/K}(u)$, and is contained in the
\hbox{rank-$1$} group $U_K$.  Therefore $\ker(\N_{L_1/K})$ has rank~$1$,
and is thus of the form $\pm u_1^\Z$ for some $u_1$ (so that if
$u_1 = x_1 + v z_1$ then $(x_1,z_1)$ is a fundamental solution
of $x_1^2 - u z_1^2 = 1$ over~$\Z[w]$).  Hence the solutions
$a'+vz \in \Z[v]$ of \eqref{eq:pell_v} constitute a finite number of
cosets of $\pm u_1^\Z$, and in each coset the condition that
$\N_{K/\Q}(z) = \pm 1$ becomes an exponential diophantine equation
in one variable.  Thus the \hbox{$p$-adic} technique applies:
extend from $u_1^\Z$ to $u_1^{\Z_p}$ for some prime~$p$,
write $\N_{K/\Q}(z)$ as a function on $\Z_p$,
and count the preimages of $\pm 1$.  As it happens, our problem
gives rise to a very favorable case of this technique:
there is only one coset, and when we choose $p=2$ we find
four preimages, each corresponding to one of our known solutions.
(In general, not all \hbox{$p$-adic} solutions come from~$\Z$, and the
spurious ones must be ruled out by further analysis.)  The details follow.

 The unit group of~$L_1$ is generated mod $\pm 1$ by $v$ and
$v^3-v^2-v = -u + (u-1)v = -(1+w) + wv$.\footnote{ We
  obtained this using the built-in function \texttt{bnfinit} in
  \texttt{gp} \cite{pari}; these units $v$ and $v^3-v^2-v$ are also
  the generators of $U_{L_1} \bmod \{\pm 1\}$ listed in the LMFDB
  \cite{lmfdb} entry 4.2.1024.1 for~$L_1$.}  Of these, $v$ has norm
$-u$ but $-(1+w) + wv$ has norm~$1$, so the kernel of $\N_{L_1/K}:
U_{L_1} \ra U_K$ consists of the powers of
\begin{equation*}
u_1 := -(1+w) + wv
\end{equation*}
and their negatives.  Note that
$\N_{L_1/K}(1+v) = \N_{L_1/K}(1-v) = -w$ and that $(1+v)/(1-v) = -(1+w) + wv$;
since this is a unit, the ideals $(1+v)$ and $(1-v)$ are the same,
so each has square $(w)$, whence the ideal $(w)$ is ramified in~$L_1$
(and thus the rational prime $(2)$ of~$\Q$ is also totally ramified
in~$L_1$, where it factors as $(1+v)^4$).  Now since the ideal $(wu) = (w)$
is prime in~$K$\/ and ramified in~$L_1$, we know that the solutions
$\alpha \in \Z[v]$ of the equation $\N_{L_1/K}(\alpha) = wu$ either form
a single coset of $\pm u_1^\Z$ or do not exist at all; and a quick search
finds the solution
\begin{equation*}
\alpha_1 = 1+w+v = u+v,
\end{equation*}
so the general solution is $\alpha = \pm \alpha_1\0 u_1^k$.
(We could also have found an initial solution by working backwards from
one of the first four lines of the list in the theorem with $\|b\| = -1$.)
These solutions must be permuted by the Galois involution of $L_1/K$, 
and indeed we compute that $u-v = -\alpha_1 u_1$, so in general the
$\Gal(L_1/K)$ conjugate of $\alpha_1\0 u_1^k$ is $-\alpha_1\0 u_1^{1-k}$.

Recall that we seek $\alpha = a'+vz$ such that $\N_{L_1/K}(\alpha) = wu$
and $\N_{K/\Q}(z) = \pm 1$.  We now know that $\N_{L_1/K}(\alpha) = wu$
is equivalent to $\alpha = \pm \alpha_1\0 u_1^k$, and that the $v$~coefficient
of $\alpha_1\0 u_1^k$ is invariant under $k \llra 1-k$.  We may thus assume that
$k$ is~even.  Moreover the choice of sign in $\alpha = \pm \alpha_1\0 u_1^k$
does not affect $\N_{K/\Q}(z)$.  We tabulate $a',z,\N_{K/\Q}(z)$ for 
the five smallest $\{k,1-k\}$ pairs, listing the even $k$\/ first
in each pair:

{\small 
\[
\begin{array}{c||c|c|c|c|c}
k & 0,1 & 2,-1 & -2,3 & 4,-3 & -4,5
\cr \hline
a' & \pm(1+w) & \pm(5+3w) & \pm(21+15w) & \pm(97+69w) & \pm (449+317w)
\cr \hline
z & 1 & -(3+2w) & 13+10w & -(63+44w) & 289+204w
\cr \hline
\N_{K/\Q}(z) & 1 & 1 & -31 & 97 & 289
\end{array}
\]
}

 Extending this calculation further suggests that if $k'
\equiv k \pmod 4$ but $k' \neq k$, then the corresponding
values of $\N_{K/\Q}(z)$ are congruent modulo $2^{\val_2(k'-k)+3}$ but
not modulo~$2^{\val_2(k'-k)+4}$, and in particular that no value
appears more than once in each congruence class $k \equiv k_0 \pmod
4$.  Since each congruence class already contains one case of
$\N_{K/\Q}(z) = 1$, this would imply that there are no others, and
thus that our list of $\|b\| = -1$ solutions is complete.  In the
remainder of this section, we prove this (in the equivalent form of
Proposition \ref{oryx} below) by extending $\N_{K/\Q}(z)$ to a
continuous function from $k \in 2\Z_2$ to~$\Z_2$, and in effect
finding the valuation of this function's derivative.

We compute that $\N_{L_1/K} (1-u_1) = 4+2w$ and $\N_{L_1/\Q} (1-u_1) = 8$,
so in particular $\val_2(1-u_1) = 3/4$ and $\val_2(1-u_1^2) = 3/2 > 1$.
Thus $u_1^{2j}$ has a \hbox{$2$-adically} convergent binomial expansion
 \begin{equation*}
u_1^{2j} = \left(1 - \left(1-u_1^2\right)\right)^j
= \sum_{n=0}^\infty (-1)^n {j \choose n} (1-u_1^2)^n
\end{equation*}
for $j \in \Z$, which extends to a continuous function $j \mapsto u_1^{2j}$
from $\Z_2$ to $\Z_2[v]$.  The same is then true of $a'$ and~$z$,
which take values in $\Z_2[w]$, and of $\N_{K/\Q}(z)$, taking values in $\Z_2$.

We first determine the function $z = z(j)$.  For $n \ge 0$, write
\begin{equation*}
(1-u_1^2)^n = r_n + v s_n
\end{equation*}
with $r_n, s_n \in \Z[w]$.  Then
\begin{align*}
\alpha_1\0 u_1^{2j}
& = \sum_{n=0}^\infty (-1)^n {j \choose n} (u+v) (r_n + v s_n)
\\
& = \sum_{n=0}^\infty (-1)^n {j \choose n}
 \bigl( (ur_n + us_n) + v (r_n + u s_n) \bigr) \quad\mbox{using \eqref{roast}}.
\label{eq:binom_rs}
\end{align*}
Thus $z(j)$ is the $v$ coefficient 
\begin{equation*}
z(j) = \sum_{n=0}^\infty (-1)^n {j \choose n} (r_n + u s_n)
  = \sum_{n=0}^\infty (-1)^n {j \choose n} t_n,
\end{equation*}
where  
\begin{equation*}
t_n := r_n + u s_n
\end{equation*}
(see Table~\ref{tab:1} for the values of $r_n, s_n, t_n$ for $0\le n \le 4$).
Therefore
\be
\label{eq:ssum}
\N_{K/\Q}(z(j)) = z(j) \, \sigma(z(j))
  = \sum_{n_1=0}^\infty \sum_{n_2=0}^\infty
 (-1)^{n_1+n_2} {j \choose n_1}{j \choose n_2} \, t_{n_1} \, \sigma(t_{n_2}).
\ee
We proceed to study how each term in the double sum changes when $j$
is replaced by some $j'$ that is \hbox{$2$-adically} close to~$j$.

\begin{table}[h!]
  \caption{ $r_n$, $s_n$, $t_n$ for $n\leq 4$}\label{tab:1}
  \[
\begin{array}{c||c|c|c|c|c}
n & 0 & 1 & 2 & 3 & 4
\cr \hline
r_n & 1 & -4 - 4w & 104 + 72w & -2080 - 1472w & \m42176 + 29824w
\cr \hline
s_n & 0 & \m4 + 2w & -64 - 48w & 1344 + 944w & -27136 - 19200w
\cr \hline
t_n & 1 & \m4 + 2w & -56 - 40w & 1152 + 816w & -23360 - 16512w
\end{array}
\]
\end{table}

We observed already that $\val_2(1-u^2) = 3/2$.
Thus
\begin{equation*}
\val_2(r_n + v s_n) = \val_2(\bigl(1-u^2)^n\bigr) = 3n/2.
\end{equation*}
We claim
\begin{lemma}
\label{addax}

Each of $r_n$ and $s_n$, and thus also $t_n$,
has valuation at least $3n/2$.
\end{lemma}

\begin{proof}

We have seen that the ideal $(1+v)^4$ of $L_1$ equals $(2)$,
so $\val_2(1+v) = 1/4$.  Since $r_n + v s_n = (r_n - s_n) + (1+v) s_n$
has valuation $3n/2$, and the image of $K$ under $\val_2$ is
$\frac12 \Z \cup \{\infty\}$, the terms $r_n - s_n$ and $(1+v) s_n$
have different valuations, so each of these valuations must be at least $3n/2$,
and then the same is true of the valuation of $r_n = (r_n-s_n) + s_n$
and Lemma \ref{addax} is proved.
\end{proof}

As for the factor ${j \choose n_1}{j \choose n_2}$ in \eqref{eq:ssum},
we show

\begin{lemma}
\label{eland}

If $n_1,n_2$ are nonnegative integers, then
\begin{equation*}
\val_2\left(
  {j' \choose n_1}{j' \choose n_2} - {j \choose n_1}{j \choose n_2}
\right)
\ge \val_2(j'-j) - n_1 - n_2 + 1
\end{equation*}
holds for all $j,j' \in \Z_2$,
with strict inequality if both $n_1$ and $n_2$ are positive.
\end{lemma}

\begin{proof}

Let $P(X) = n_1! n_2! {X \choose n_1}{X \choose n_2} \in \Z[X]$.
Then
\begin{equation*}
  {j' \choose n_1}{j' \choose n_2} - {j \choose n_1}{j \choose n_2}
  = \frac{P(j') - P(j)}{n_1! \, n_2!}.
\end{equation*}
The numerator is a multiple of $j'-j$, and thus has valuation at least
$\val_2(j'-j)$.  For the denominator, we use the inequality
\be
\label{eq:v(n!)}
  \val_2(n!) \le n,
\ee
which is valid for all nonnegative integers~$n$,
and strict for $n>0$.  (The difference is the number of 1s
in the binary representation of~$n$; this is a known consequence of
the general formula for $\val_p(n!)$ in terms of the \hbox{base-$p$}
representation of~$n$.)  This completes the proof of Lemma \ref{eland};
it might seem that $(n_1,n_2)=(0,0)$ is an exception, but in this case
${j' \choose n_1}{j' \choose n_2} - {j \choose n_1}{j \choose n_2} = 1 - 1 = 0$
for all $j,j'$ and there is nothing to prove.
\end{proof}

Combining Lemmas \ref{addax} and \ref{eland} gives

\begin{lemma}
\label{impala}

If $n_1,n_2$ are nonnegative integers, then
\begin{equation*}
 \val_2\left(
  \left[
     {j' \choose n_1}{j' \choose n_2} - {j \choose n_1}{j \choose n_2}
  \right]
   \, t_{n_1} \, \sigma(t_{n_2})
 \right)
 \ge \val_2(j'-j) + \frac{n_1 + n_2}{2} + 1
\end{equation*}
holds for all $j,j' \in \Z_2$,
with strict inequality if both $n_1$ and $n_2$ are positive.\qed
\end{lemma}

Recall that our aim is to prove 

\begin{proposition}
\label{oryx}

If $j',j \in \Z_2$ with $j' \equiv j \bmod 2$ then
\begin{equation*}
 \val_2(\N_{K/\Q}\0 z(j') - \N_{K/\Q}\0 z(j)) = \val_2(j'-j) + 4.
\end{equation*}
\end{proposition}
\noindent (We took $k = 2j$, so if we likewise set $k' = 2j'$ then
$k' \equiv k \pmod 4$ makes $j' \equiv j \pmod 2$, and
$\val_2(k'-k) + 3 = \val_2(j'-j) + 4$.)

\begin{proof}

By Lemma \ref{impala} it suffices to prove this with
the double sums \eqref{eq:ssum} for $\N_{K/\Q}\0 z(j')$ and $\N_{K/\Q}\0 z(j)$
replaced by finite sums over $n_1 + n_2 < 6$ together with
$(n_1,n_2) = (0,6)$ and $(6,0)$.  We next check that
it is enough to consider $n_1,n_2 \in \{0,1,2\}$ together with
$(n_1,n_2) = (0,4)$ and $(4,0)$.

Combining the $(n_1,n_2)$ and $(n_2,n_1)$ terms when $n_1 \neq n_2$,
we rewrite~\eqref{eq:ssum} as
\begin{align*}
\label{eq:ssum1}
\N_{K/\Q}(z(j))
  = &
  \sum_{n=0}^\infty {j \choose n}^2 \, \N_{K/\Q}(t_n)
\\
&  + \sum_{n_1=0}^\infty \sum_{n_2=n_1+1}^\infty
   (-1)^{n_1+n_2} {j \choose n_1}{j \choose n_2}
     \, \Tr_{K/\Q}(t_{n_1} \, \sigma(t_{n_2})).
\end{align*}
We computed $r_n,s_n,t_n$ for $0\leq n \le 8$; 
Table~\ref{tab:1} lists the values for $n \le 4$.

 Each $t_n$ turns out to have valuation $3n/2$, attaining equality in
\linebreak Lemma~\ref{addax}, so $\val_2(\N_{K/\Q}(t_n)) = 3n$.  We next tabulate
$\val_2\bigl( \Tr_{K/\Q}(t_{n_1} \, \sigma(t_{n_2})) /\allowbreak (n_1! \, n_2!) \bigr)$
for $n_1 < 3$ and $ n_1 < n_2 \le 6$, which includes all cases of
$n_1 + n_2 \le 6$ with $n_1 < n_2$, and thus confirms our claim that
the terms with $n_1,n_2 \le 2$ or $(n_1,n_2) = (0,4)$ and $(4,0)$ suffice:
\[
\begin{array}{c||c|c|c|c|r|r}
    n_2 & 1 & 2 & 3 & 4 & 5 & 6
\cr \hline
n_1 = 0 & 3 & 3 & 7 & 4 & 6 & 6
\cr
n_1 = 1 &   & 6 & 6 & 6 & 7 & 12
\cr
n_1 = 2 &   &   & 7 & 6 & 10 & 8
\end{array}
\]
\noindent (Note that we are using here the actual $\val_2(n_i!)$,
not the upper bound~\eqref{eq:v(n!)}.)

 Using only the terms for which
\[
\val_2(\N_{K/\Q}(t_n)/n!^2) \quad \text{or} \quad
\val_2\left( \frac{\Tr_{K/\Q}(t_{n_1} \, \sigma(t_{n_2})}{n_1! \, n_2!} \right)
\]
is at most~$4$ (and thus omitting also the pair $(n_1,n_2)=(1,2)$ and $(2,1)$), we compute
\begin{equation*}
 \val_2(\N_{K/\Q}\0 z(j))
  = 1 + 2^4\left( 733j - \frac{4051}{3} j^2 + 740 j^3 - \frac{368}{3} j^4 \right)
  + \eta(j)
\end{equation*}
for some function $\eta: \Z_2 \to 2^4 \Z_2$ such that
$\val_2(\eta(j') - \eta(j)) > \val_2(j'-j) + 4$ for all distinct
$j,j' \in \Z_2$.
So it remains to check that the difference between the values of
$1 + 2^4( 733X - \frac{4051}{3} X^2 + 740 X^3 - \frac{368}{3} X^4 )$
at $X=j'$ and $X=j$ has valuation exactly $\val_2(j'-j) + 4$
provided $j' \equiv j \pmod 2$.
The term $1$ does not change; the term $2^4 733 X$ changes by
$2^4 733 (j'-j)$, which has the desired valuation $\val_2(j'-j) + 4$;
the term $2^4 \frac{4051}{3} X^2$ changes by
$2^4 \frac{4051}{3} (j'-j) (j'+j)$, which has valuation strictly greater
than $\val_2(j'-j) + 4$ because $\val_2(j'+j) > 0$ by the assumption
$j' \equiv j \bmod 2$; and the change in each of the remaining terms
$2^4 740 X^3$ and $-2^4 \frac{368}{3} X^4$ is $2^4(j'-j)$ times
some multiple of~$2$, and thus has valuation $> \val_2(j'-j) + 4$
for all $j,j' \in \Z_2$.
\end{proof}

 This proves Proposition \ref{oryx},
and thus completes at last our proof that our list of
$\|b\| = -1$ solutions is complete, concluding the proof of
Theorem~\ref{tips}. 

\begin{remark1}
We could also solve the case $\|b\| = -2$ using this method.
Here $b = \pm w z^2$ so $wu = {a'}^2 + b = {a'}^2 \pm w z^2$
for some unit~$z$.  Again the two choices of sign are equivalent
under Galois conjugation.  This time we are working in the quartic field
$L_2 = \Q(\root 4 \of 2\,) = \Q(v)$ where $v^2 = w$, number field 4.2.2048.1 in
\cite{lmfdb}; the units are generated
mod~$\pm 1$ by $1+v$ and $1-v$, and the kernel of $\N_{L_2/K}\colon U_L \ra U_K$
is $\pm u_2^\Z$ where $u_2 = -(1+v)/(1-v) = (3+2w) + (2+2w) v$.
The solutions of $\N_{L_2/K} \alpha = wu$ in $\Z[v]$ are
$\pm \alpha_2\0 u_2^k$ where $k\in\Z$ and $\alpha_2 = u(w-v)$,
with $\pm \alpha_2\0 u_2^{1-k}$ the $\Gal(L_2/K)$ conjugate of
$\pm \alpha_2\0 u_2^k$.  This time there is only one pair of solutions
$\alpha = \pm a' \pm v z$ with $\N_{K/\Q} z = \pm 1$, namely
$\pm uw \pm uv$ itself (with $k=0,1$).  Again this can be proved
by fixing the parity of~$k$ (which loses no generality thanks to the
symmetry $k \llra 1-k$) and regarding $\alpha_2\0 u_2^k$ as a
$\Z_2[v]$-valued function of $k \in 2\Z_2$.  We spare 
the details, for both our sake and the reader's.
\end{remark1}

\begin{corollary}
\label{pot}
There are exactly ten $\Z[\sqrt{2}]$-PCFs for $\sqrt{2+\sqrt{2}}$
of type $(1,2)$.
They are
\begin{align*}
\label{rind}
  \sqrt{2+\sqrt{2}}
  & =  [1+\sqrt{2},\,\overline{-2,\,2+2\sqrt{2}}]\\
  &= [1,\,\overline{2\sqrt{2}-2,\,2}]\\
  & =  [-1+2\sqrt{2},\,\overline{26+18\sqrt{2},\,-2+4\sqrt{2}}]\\
  & =  [5+3\sqrt{2},\,\overline{62-44\sqrt{2},\,10+6\sqrt{2}}]\\
  & =  [\sqrt{2},\,\overline{2,\,2\sqrt{2}}]\\
  & =  [2+\sqrt{2},\,\overline{2-2\sqrt{2},\,4+2\sqrt{2}}]\\
  & =  [2,\,\overline{-4-2\sqrt{2},\,4}]\\
  & =  [2+2\sqrt{2},\,\overline{8-6\sqrt{2},\,4+4\sqrt{2}}]\\
  & =  [-182+130\sqrt{2},\,
    \overline{-123656-87438\sqrt{2},\,-364+260\sqrt{2}}]\\
  & =  [442+312\sqrt{2},\,\overline{-298532+211094\sqrt{2},\,884+624\sqrt{2}}].
\end{align*}
\end{corollary}

\begin{proof}

Again we check convergence using Algorithm \ref{con}.

\end{proof}
\noindent For the last continued fraction $[b_1, \overline{a_1, a_2}]$
we have 
$a_1\alpha_2-b_1a_1-1\approx 0.995825$, which means that every extra
significant digit requires $\approx -2/(2\log_{10}(0.995825))\approx 550 $
more convergents by Remark \ref{pudu}\ref{pudu3}.

\begin{remark1}
          In keeping with Proposition \ref{stunts}, the number $21$ of
          $\OO_1$-points on $E(2+\sqrt{2})$ in Theorem
          \ref{tips} (which is the same as the number of $\OO_1$-points
          on $V(2+\sqrt{2})_{1,2}'$ by Remark \ref{pinto})
          is congruent to $1\bmod4$.  Likewise also in keeping
          with Proposition \ref{stunts}, the number $10$ of 
          $\OO_{1}$-PCFs of $\alpha_2$ of type $(1,2)$
          is divisible by $2$.
        \end{remark1}

\appendix
\section{Proofs of Convergence Results}

As advertised we now prove the results in Section \ref{convergence}.
By ``converge'' we mean converge in the $\BP^1(\C)$ metric,
so in particular we say the limit exists even if it is converging to infinity.

\begin{proposition}\label{giraffe}
Let $A=\left[\begin{smallmatrix}a&b\\
c&d\end{smallmatrix}\right],ad-bc=\varepsilon=\pm1$,
and $z\in\BP^1(\C)$.
Let
\[w=\lim_{n\to\infty}\oA^n(z)\]
if the limit exists.
Let $\lambda_\pm$ be the eigenvalues of $A$ chosen so that
$|\lambda_+|\ge1\ge|\lambda_-|$.
If $A\ne\pm\sqrt{\varepsilon} I$ let
\[\beta_\pm=\beta_\pm(A)=\frac{\lambda_\pm-d}c=\frac{b}{\lambda_\pm-a}
\left(=\frac{a-\lambda_\mp}c=\frac{b}{d-\lambda_\mp}\right)\in\BP^1(\C)
\]
where we take whichever expression is not the indeterminate $0/0$.
\begin{enumerate}
\item\label{giraffe1} If $A=\pm\sqrt{\varepsilon} I$, then $w=z$ for every $z$.
\item\label{giraffe2} If $\left|a+d\right|<2$, $(a+d)\sqrt{\varepsilon}\in\R$,
and $z=\beta_\pm$, then $w=z$.
\item\label{giraffe3} If $A\ne\pm\sqrt{\varepsilon} I$ and
$a+d=\pm2\sqrt{\varepsilon}$, then $w=\beta_+=\beta_-$ for every $z$.
\item\label{giraffe4} Suppose $\left|a+d\right|>2$ or
$(a+d)\sqrt{\varepsilon}\not\in\R$.
If $z=\beta_-$, then $w=\beta_-$.
\item\label{giraffe5} The limit
does not exist if and only if
$\left|a+d\right|<2$, $(a+d)\sqrt{\varepsilon}\in\R$, and $z\ne \beta_\pm$.
\item\label{giraffe6} Suppose $|a+d|>2$ or $(a+d)\sqrt{\varepsilon}\not\in\R$.
If $z\ne \beta_-$, then $w=\beta_+$.
\end{enumerate}
The real dimension count of the choice $(A,z)$ for the six cases is,
respectively, $2,5,6,6,7,8$.
In particular, case \textup{\eqref{giraffe6}} is the generic case.
\end{proposition}

\begin{proof}
First note that if $A=\pm\sqrt{\varepsilon} I$,
then $\oA(z)=z$ for every $z$,
so $\lim_{n\to\infty}\oA^n(z)=z.$
So assume $A\ne\pm\sqrt{\varepsilon} I$ for the remainder of the proof,
so that in particular $\beta_\pm$ is defined.
Since $\oA(\beta_\pm)=\beta_\pm$,
\[\lim_{n\to\infty}\oA^n(\beta_\pm)=\beta_\pm.\]

As in the proof of Proposition \ref{ratify}
$A^n=t_{n-1}A-\varepsilon t_{n-2}I$ where
\[t_n=\begin{cases}
\frac{\lambda_+^{n+1}-\lambda_-^{n+1}}{\lambda_+-\lambda_-}&\text{if }
\lambda_+\ne\lambda_-,\\
(n+1)\lambda_+^n&\text{if }\lambda_+=\lambda_-.
\end{cases}\]
Hence if $t_{n-1}\ne0$,
\[\oA^n(z)=\frac{\left(a-\varepsilon\frac{t_{n-2}}{t_{n-1}}\right)z+b}{cz+\left(d-\varepsilon\frac{t_{n-2}}{t_{n-1}}\right)}.\]
If $\left|\lambda_+\right|>1>\left|\lambda_-\right|$ or $\lambda_+=\lambda_-=\pm\sqrt{\varepsilon}$
then
\[
\lim_{n\to\infty}\oA^n(z)=\frac{\left(a-\frac{\varepsilon}{\lambda_+}
\right)z+b}{cz+\left(d-\frac{\varepsilon}{\lambda_+}\right)}
=\frac{\left(a-\lambda_-\right)z+b}{cz+\left(d-\lambda_-\right)}=\beta_+
\]
provided $z\ne(\lambda_--d)/c=b/(\lambda_--a)=\beta_-$.
If $\left|\lambda_+\right|=1=\left|\lambda_-\right|$ but 
$\lambda_+\ne\lambda_-$ then
the map $z\to\oA(z)$
is conjugate to $z\to e^{i\theta}z$ for some $\theta\ne 2\pi m$, and
$\lim_{n\to\infty}e^{ni\theta}z$ does not exist unless $z$ is one of the fixed
points $z=0,\infty$.
\end{proof}

\begin{remark1}
  \begin{enumerate}
\item The indeterminate $0/0$ occurs in the definition of the $\beta_\pm$
if and only if $bc=0$ if and only if one of $\beta_\pm$ is $0$ or $\infty$.
Both expressions for, say, $\beta_+$ are the indeterminate $0/0$ if and only if
$A$ is a multiple of the identity, which is excluded from the definition.
\item Cases \ref{giraffe1} and \ref{giraffe3} correspond to
$\lambda_+=\lambda_-$, Cases \ref{giraffe2} and \ref{giraffe5} to
$|\lambda_+|=|\lambda_-|=1,\lambda_+\ne\lambda_-$, and Cases \ref{giraffe4}
and \ref{giraffe6} to $|\lambda_+|>1>|\lambda_-|$.
In Cases \ref{giraffe4} and \ref{giraffe6}, $\beta_+(A)$ is called the
{\em attractive} point and $\beta_-(A)$ the {\em repulsive} point.
In Case \ref{giraffe3} $\beta_+(A)=\beta_-(A)$ is also called attractive.
In Cases \ref{giraffe2} and \ref{giraffe5} both points are called
{\em indifferent} .
\item An alternate and more direct definition of $\beta_\pm(A)$ is
\[\beta_\pm(A)=\frac{a-d\pm f_{-\varepsilon}(a+d)}{2c}=
  \frac{2b}{d-a\pm f_{-\varepsilon}(a+d)}\in\BP^1(\C)\]
if $A\ne\pm\sqrt{\varepsilon}I$, where we take $\beta_\pm(A)$ to be whichever
expression is not an indeterminate $0/0$,
\[f_\pm(z)=z\sqrt{1\pm4/z^2}\]
for $z\ne0$ with the usual principal value for the square root,
$f_+(0)=2$, and $f_-(0)=2i$.
The key property of $f_\pm$ is that $|z+f_\pm(z)|>2$ everywhere but the branch
cut, proved in Proposition \ref{pronghorn} below.
Note $f_\pm(-z)=-f_\pm(z)$ for $z\ne0$, which is necessary for $\beta_\pm$ to
depend only on $\oA$.
  \end{enumerate}
\end{remark1}


\begin{proposition}
\label{pronghorn}
The function $f_-(z)$ is holomorphic on the open set that is the complement
of the branch cut $-2\le z\le2$ where $z$ is real.
The function $f_+(z)$ is holomorphic on the open set that is the complement
of the branch cut $-2\le iz\le2$ where $z$ is imaginary.
Furthermore $\left|z+f_\pm(z)\right|=2$ on the branch cuts and
$\left|z+f_\pm(z)\right|>2$ everywhere else.

\end{proposition}

\begin{proof}
The holomorphicity and branch cuts follow from the definition of the branch
cut of square root.
The last statement is equivalent to
\[\left|1+\sqrt{1-z^2}\right|\begin{cases}=\left|z\right|
\text{ if }z\text{ is real and }z^2\ge1,\\
>\left|z\right|\text{ otherwise.}
\end{cases}\]
$\left|1+\sqrt{1-z^2}\right|^2=1+\left|z^2-1\right|+2\Rea(\sqrt{1-z^2})$ and the real part of a square
root is always $\ge0$ by definition, so the result now follows from the triangle
inequality.
\end{proof}

\begin{corollary}\label{okapi}
Let $A$ and $z$ be as in Proposition~\textup{\ref{giraffe}},
and $\left[\begin{smallmatrix}e&f\\g&h\end{smallmatrix}\right]=BAB^{-1}$ for some invertible $B$.
Then $\overline{B}(\beta_\pm(A))=\beta_\pm(BAB^{-1})$.
Furthermore, let
\[
w=\lim_{n\to\infty}\overline{B}\oA^n(z)
\]
if the limit exists.
\begin{enumerate}
\item If $A=\pm\sqrt{\varepsilon} I$, then $w=\overline{B}(z)$ for every $z$.
\item If $\left|a+d\right|<2$, $(a+d)\sqrt{\varepsilon}\in\R$, and $z=\beta_\pm(A)$,
then $w=\overline{B}(z)$.
\item If $A\ne\pm\sqrt{\varepsilon} I$ and $a+d=\pm2\sqrt{\varepsilon}$, then
$w=\overline{B}(\beta_+(A))=\overline{B}(\beta_-(A))$ for every $z$.
\item Suppose $\left|a+d\right|>2$ or $(a+d)\sqrt{\varepsilon}\not\in\R$.
If $z=\beta_-(A)$, then $w=\overline{B}(\beta_-(A))$.
\item The limit $\lim_{n\to\infty}\overline{B}\oA^n(z)$ does not exist if and
only if $\left|a+d\right|<2$, \mbox{$(a+d)\sqrt{\varepsilon}\in\R$}, and $z\ne \beta_\pm(A)$.
\item Suppose $\left|a+d\right|>2$ or $(a+d)\sqrt{\varepsilon}\not\in\R$.
If $z\ne \beta_-(A)$, then $w=\overline{B}(\beta_+(A))$.
\end{enumerate}
\end{corollary}

\begin{proof}
 Simply note that
$\overline{B}\oA^n(z)=(\overline{BAB^{-1}})^n(\overline{B}z)$
and apply Proposition~\ref{giraffe}.
\end{proof}

\begin{corollary}\label{brocket}
With notation as in Corollary \textup{\ref{okapi}}, let
\[w=\lim_{n\to\infty}\overline{B}\oA^n(\infty)\]
if the limit exists.
\begin{enumerate}
\item If $A=\pm\sqrt{\varepsilon} I$, then $w=\overline{B}(\infty)$.\label{brocket1}
\item If $\left|a+d\right|<2$, $(a+d)\sqrt{\varepsilon}\in\R$, and $c=0$,
then $w=\overline{B}(\infty)$.\label{brocket2}
\item If $A\ne\pm\sqrt{\varepsilon} I$ and $a+d=\pm2\sqrt{\varepsilon}$, then
$w=\overline{B}(\beta_+(A))=\overline{B}(\beta_-(A))$.\label{brocket3}
\item Suppose $\left|a+d\right|>2$ or $(a+d)\sqrt{\varepsilon}\not\in\R$.
If $c=0$ and $\left|a\right|<1<\left|d\right|$, 
then $w=\overline{B}(\infty)$.\label{brocket4}
\item The limit $\lim_{n\to\infty}\overline{B}\oA^n(\infty)$ does not exist if and only if
$\left|a+d\right|<2$, $(a+d)\sqrt{\varepsilon}\in\R$, and 
$c\ne0$.\label{brocket5}
\item Suppose $\left|a+d\right|>2$ or $(a+d)\sqrt{\varepsilon}\not\in\R$.
If $c\ne0$ or $\left|a\right|>1>\left|d\right|$, 
then $w=\overline{B}(\beta_+(A))$.\label{brocket6}\hfill\qed
\end{enumerate}
\end{corollary}

 We apply Corollary \ref{brocket} to
$B=B_j=M([b_1,\ldots, b_N, a_1, \ldots  a_{j}])$ and
\[A=A_j=M([a_{j+1},\ldots,a_{k},a_{1},\ldots,a_{j}])\text{ for }j=0,\ldots,k-1.\]
By periodicity, $B_0^{-1}B_jA_j=A_0B_0^{-1}B_j$, and thus
$B_jA_jB_j^{-1}=B_0A_0B_0^{-1}=E$ by Equation \eqref{fed1}.

\renewcommand{\thesection}{\arabic{section}}
\setcounter{section}{4}
\setcounter{theorem}{2}
\begin{theorem}

  \label{punchy}
Let $P=\PP$ be
a PCF.
Then the value $\ob(P)$ exists if and only if \textbf{none} of
the following three conditions is satisfied:
\begin{enumerate}
\item\label{bactrian1'}
$E(P)=\pm i^k I$.
\item\label{bactrian2'}
With the $a_i$ periodic as in \eqref{periodic}, 
\[
M([a_{j+1},\ldots,a_{k+j}])_{21}=0\mbox{ and }
\left|M([a_{j+1},\ldots,a_{k+j}])_{22}\right|>1
\]
for some $j=0,1,\ldots,k-1$.
\item\label{bactrian3'}
$\Tr(E(P))^2\in\R$ and $0\le(-1)^k\Tr(E(P))^2<4$.
\end{enumerate}
If it converges, the value $\ob(P)=\beta_+(E(P))$.
\end{theorem}

\begin{proof}
This is just an application of Corollary \ref{brocket} to $B_j$ and $A_j$,
for $j=0,\ldots,k-1$,
with the complication that the limit for each $j$ must be the same.
As noted above the $A_j$ are all in the same conjugacy class,
which determines whether we're in one of four subsets of cases of
Corollary~\ref{brocket}:
Case~\ref{brocket1},
Cases~\ref{brocket2}/\ref{brocket5},
Case~\ref{brocket3}, or Cases~\ref{brocket4}/\ref{brocket6}.
Because consecutive convergents cannot be equal,
$\overline{B_j}(\infty)\ne \overline{B_{j+1}}(\infty)$,
and therefore Case \ref{brocket1} (Theorem \ref{punchy}\ref{bactrian1'}),
Case~\ref{brocket4} (Theorem \ref{punchy}\ref{bactrian2'}, which has to be checked for each $j$ since
it is not conjugation invariant), and
Cases \ref{brocket2}/\ref{brocket5} (Theorem \ref{punchy}\ref{bactrian3'}) are excluded.
So if the limit exists,
we have to be in Case \ref{brocket3} or \ref{brocket6},
and the limit is $\beta_+(E)=\overline{B_j}(\beta_+(A_j))$ for every~$j$.
\end{proof}
\noindent If Theorem \ref{punchy}\ref{bactrian2'} is satisfied, but not \ref{bactrian1'}
and \ref{bactrian3'}, then for each $j\bmod k$ the limit
\[\lim_{n\to\infty}[b_1,\ldots, b_N, a_1, \ldots, a_{j+nk}]\]
exists, and most are the same, but there is at least one pariah $j$ for which
the limit exists but is different from the others.
\renewcommand{\thesection}{\Alph{section}}
\setcounter{section}{1}
\setcounter{theorem}{5}
\begin{corollary}\label{hungry}
   Let $P^*=[b_1, \ldots, b_{N},0,
    \overline{-a_{k},\ldots,-a_{2}, -a_{1}}]$ be the PCF
   dual to $P$.  Then the value $\ob(P^*)$ exists if
  and only if \textbf{none} of the following three conditions is
  satisfied:
\begin{enumerate}
\item
$E(P)=E(P^*)^{-1}=\pm\sqrt{(-1)^k}I$.
\item
With the $a_i$ periodic as in \eqref{periodic},
\[
M([a_{j+1},\ldots,a_{k+j}])_{21}=0\mbox{ and }
\left|M([a_{j+1},\ldots,a_{k+j}])_{22}\right|<1
\]
for some $j=0,1,\ldots,k-1$.
\item
$\Tr(E(P))^2\in\R$ and $0\le(-1)^k\Tr(E(P))^2<4$.
\end{enumerate}
If it converges, the value $\ob(P)=\beta_-(E(P))=\beta_+(E(P^*))$.
\end{corollary}

\begin{proof}
 By \eqref{sitatunga} $E(P^*)=E(P)^{-1}$, and hence
$\Tr(E(P^*))=(-1)^k\Tr(E(P^*))$.
If $(A_j)_{21}=0$, then
$(A_j^{-1})_{11}=(-1)^k(A_j)_{22}$ and
$(A_j^{-1})_{22}=(-1)^k(A_j)_{11}$.
\end{proof}
\begin{remark1}
Notice that the conditions in Corollary \ref{hungry} are identical to
the conditions in Theorem \ref{punchy} except that the inequality
in Theorem \ref{punchy}\ref{bactrian2'} is reversed.
\end{remark1}

\bibliographystyle{plain}
\bibliography{PCFCa}
\end{document}